\documentclass[article]{mesa-NS}
\usepackage{times,mathptm}
\usepackage{amsmath,amsfonts,amssymb}
\usepackage{graphicx,graphics,epsfig}
\usepackage{fancyhdr,fancybox}
\usepackage{verbatim}
\usepackage{color}
\usepackage{longtable}

\jvol{}
\jnum{}
\jyear{2017}

\numberwithin{equation}{section}
\numberwithin{theorem}{section}
\numberwithin{definition}{section}
\numberwithin{lemma}{section}
\numberwithin{corollary}{section}
\numberwithin{proposition}{section}
\numberwithin{example}{section}
\numberwithin{remark}{section}

\begin{document}
\label{pageinit}

\date{}

\title{ Duality structure, asymptotic analysis and emergent fractal sets  }

\author{ Dhurjati Prasad Datta$^{1^\star}$, Soma Sarkar  $^1$}

\maketitle

\noindent $^1$ Department of Mathematics, University of North Bengal \\
Siliguri,West Bengal, Pin: 734013, India \\

\noindent $^\star$ \emph{Corresponding Author}. E-mail: dp${_-}$datta@yahoo.com

\markboth{Dhurjati Prasad Datta, Soma Sarkar }{Duality structure, asymptotic analysis and emergent fractal sets}



\begin{abstract}
A new,  extended nonlinear framework of the ordinary real analysis
incorporating a novel concept of {\em duality structure}  and its applications into
various nonlinear dynamical problems is presented. The duality structure is an asymptotic property that should affect the late time
asymptotic behaviour of a nonlinear dynamical system in a nontrivial way leading naturally to signatures generic to a complex system. We argue
that the present formalism would offer a natural framework to understand the abundance of complex systems in natural, biological, financial and related problems. We show that the power law attenuation of a dispersive, lossy wave equation, conventionally deduced from fractional calculus techniques, could actually arise from the present asymptotic duality   structure. Differentiability on a Cantor type fractal set is also formulated.
\end{abstract}


\section{Introduction}
The present paper deals with a new,  extended nonlinear framework of the ordinary real analysis
incorporating a novel concept of {\em duality structure} (\cite{dpr1}-\cite{dpp}) and its applications into
various nonlinear dynamical problems. The duality structure is an asymptotic property that should affect the late time
asymptotic behaviour of a nonlinear dynamical system in a nontrivial way leading naturally to signatures generic to a complex system. We argue
that the present formalism would offer a natural framework to understand the abundance of complex systems in natural, biological, financial and related problems.

To put the ideas into right perspective, let us recall the the framework of ordinary  real analysis may be considered  {\em linear} because the fundamental notions of limit, differentiability etc are defined on the basis of linear shift
increments in association with the usual (or any equivalent) metric. The duality in the extended analysis
is an asymptotic notion that is realized in an arbitrarily small neighbourhood of a
point $x_0$ (say, for instance, 0) in the real line $\bf R$.
We postulate that in such an arbitrarily small vicinity, when linear increments are vanishingly small, there
still exists a swarm of discretely distributed, no where dense set  (of Lebesgue measure zero) of points, transition between them being
mediated by a {\em generalized duality transformation} of the form $x\mapsto \tilde X^{-1}=|x|^s, \ s>0$ for $0<|x|<1$ and $\tilde X>1$. Although, some of these ideas are already presented by the first author and co-workers (\cite{dpr1}-\cite{dpp}), the present paper may be seen as an effort in putting the concepts of duality structure and duality induced smooth jump mode in a more rigorous setting in the context of an asymptotic analysis.

We show that the said duality structure does indeed {\em exist} in  an asymptotic neighbourhood when the standard definition
of the Cantor completion of the rational number field is modified first by introducing the notion of a {\em scale}
and   then  by extending the usual metric with an asymptotic,
{\em duality invariant} component, known as the {\em asymptotic
visibility metric}, that {\em acts} only in an arbitrarily small neighbourhood of 0, in association with the scale.
As a variable $x$ approaches the point $x_0$, the ordinary  limiting value of the Lebesgue measure $|x-x_0|$
vanishes. However, an associated nonlinear measure, induced from the duality invariant visibility metric, remains
nonzero (with a {\em finite} real value), thus making  a room for a non-archimedean extension \cite {na} of  the ordinary real line $\bf R$ to a structured (soft) real number system $\bf R^*$ accommodating  the said  duality structure. This opens up the possibility of formulating  a novel concept of duality (jump) differentiability and hence duality invariant {\em jump differential calculus} for a large class of continuous but irregular (non-differentiable) functions. The extension of duality structure in higher dimensional spaces $\bf R^n$ or $C$ should not be difficult, except for obvious technicalities, and  will be taken up separately.

There are several motivations. Over the past decades it becomes clear that conventional
analytic formalism based on ${\bf R}^n$ calculus, standard linear and nonlinear functional analysis, theory of nonlinear differential equations etc are fundamentally deficient \cite{crutch,west,tsallis} in analyzing and deriving right scaling properties of emergent (i.e. evolutionary) complex structures in various natural and other real world problems such as origin and proliferation of biological  structures (systems) \cite{west}, financial modeling \cite{mandel1}, turbulence \cite{fris}, large scale structure formation in cosmology \cite{cos}, meteorological predictions \cite{meteor}, to name a few. Over the past couple of decades, applications of fractional calculus techniques are gaining significance in complex systems \cite{castil, klafter,tar}. Analysis and differential equations on fractals are studied in \cite{kig, difr,parvate}. Another parallel approach in complex system studies is advocated on the basis of nonextensive $q$-statistical mechanics \cite{tsallis}. The duality enhanced extension of the ordinary analysis should throw new light into traditional analytic techniques of analysis as a whole and in the theory of differential equations in particular. One expects to discover new duality supported asymptotics that can have,  not only  nontrivial applications into nonlinear differential systems, but more importantly, may initiate a radical shift in our perceptions and understanding of the origin and proliferation of complex systems , as well as, their efficient control and management. In short, the present formalism might be considered to offer an alternative approach towards complex system studies.

Fractal sets, as irregular subsets of ${\bf R}^n$, are prototypes of complex systems \cite{mandel2,falc}. To authors' view, a genuine dynamical theory  of the origin of fractal sets is still not available in a general setting. Fractal and multifractal sets appear predominantly in turbulent flows of fluid and plasma as well as in the chaotic attractors of non-autonomous deterministic system. Stated more precisely, asymptotic evolutionary states of  nonlinear systems are known to get arrested in turbulent or chaotic attractors on which the system orbit executes erratic motion on a lower dimensional, bounded fractal or multifractal subset of the phase space. As an example, let us recall the onset of chaos in the logistic map as the asymptotic state of period doubling route as the control parameter $\lambda$, say, approaches the threshold or critical value $\lambda_\infty$ \cite{falc}. The origin of erratic motion of iterates for sufficiently large $n$ is due to a dynamical selection of a zero or positive measure Cantor subset of the interval [0,1] as the asymptotic chaotic attractor of the map. The singular character of chaotic or turbulent attractors of nonlinear systems poses major obstructions in the smooth extension of the underlying differentiable structure of ${\bf R}^n$ onto a lower dimensional chaotic set. Even as there have been major advances in exploring scaling and geometric properties of chaotic attractors over past decades \cite{west, tsallis}, new analytic framework equipped with duality enhanced asymptotic scaling is likely to have significant applications in further exploration of nonlinear  dynamics.

A hallmark of complex system is the presence of many  independent components over different  scales, all of which might interact  in a collective and cooperative manner leading to an {\em emergent} evolutionary process \cite{crutch,west,tsallis}. The word emergent generally means a fundamentally new level of system property that can not simply be understood from those of the basic constituents. At a simple geometric level, the decomposition of a straight line segment into a measure zero (self similar) Cantor set \cite{falc} is an evolutionary limiting process of an iterated function system. The vanishing of the Lebesgue measure of the original line segment and subsequent realization of a nonvanishing Hausdorff measure \cite{falc,roger} can be interpreted as an emergent process. In other words, the transformation of the finite Lebesgue measure of the initial line segment into the corresponding Hausdorff measure of the limiting Cantor set may be considered to be an example of emergence.

In the present paper, we, however, interpret emergence specially in the context of the duality structure. In the presence of duality structure the ordinary real number system $\bf R$ gets extended into a non-archimedean space \cite{na, robin2}, $\bf R^*$, so that the linear neighbourhood of a point, say, 0, of the form $[-\delta,\delta]$ in the asymptotic limit $\delta\rightarrow 0$ gets extended into a Cantor-like fractal set $C\subset [0,,\tilde\delta], \ \tilde\delta \approx \delta\log\delta^{-1}$. Such extensions of the linear neighbourhood structure of $\bf R$ into a totally disconnected, measure zero Cantor set like structure is called  {\em asymptotic prolongation} of a point set $\{x\}, \ x\in \bf R$. Of course, there is an arbitrariness in the choice of the totally disconnected prolongation. To facilitate a unique choice in a given nonlinear dynamical problem, we state a {\em fundamental selection principle} that roughly goes as follows: {\em Every nonlinear complex system  selects naturally a unique duality enhanced soft model $\bf R^*$ of the real number system, that is determined by the  characteristic scales of the system,  so as to admit a richer differentiability structure awarded by the corresponding renormalized effective variable (measure) $X$ on the associated  prolongation set $\bf O$.}  Conversely, {\em every possible soft extension $\bf R^*$ is supposed to provide a natural smooth structure corresponding to a unique (equivalence) class of nonlinear systems.} A formal proof of this statement will be undertaken elsewhere. In this paper, we present some applications of this selection principle in certain emergent phenomena as explained below.

Once a right choice of the soft extension and the associated  asymptotic prolongation set are effected, the said prolongation is shown to be naturally  equipped with a unique renormalized effective variable of the form $X={\tilde x}^{\beta(\tilde x)}$ for $\tilde x\in [0,\tilde\delta]$ that corresponds to the singular Cantor function measure on the Cantor set concerned \cite{cantor}. More importantly, this renormalized variable $X$ offers itself as a {\em uniformizing } variable on the said totally disconnected prolongation set leading to a {\em generalized} smooth (differentiable) structure, analogous to the self consistent fractal continuum model of Ref.\cite{balank}, and hence is expected to have novel applications in the analysis of a complex nonlinear system.

 We now call {\em a subset of $\bf R^*$ an emergent set if it is defined as the limit set of a limiting process respecting the duality structure}. As a consequence, the asymptotic prolongation of the point set $\{0\}$ is an emergent Cantor set. It also follows that any fractal subset of $\bf R$, viewed from $\bf R^*$ is also emergent in the present sense. Moreover, any smooth subset i.e., for example, a curve, of $\bf R$ can be seen to inherit the emergent property of the basic prolongation set of $\bf R^*$. Asymptotic prolongation of a point set into a totally disconnected, zero Lebesgue measure set, equipped with the (generalized) smooth structure inherited from the uniformizing (continuum) effect of the Cantor's (fractal/multifractal) measure, {\em presents a novel framework for a dynamical interpretation of fractals/multifractals}.

As pointed out already, some applications of the present approach have already been made previously, viz,  in an  attempt of  formulating  an analytic framework for a Cantor like fractal set \cite{dpr1,dpr2, dpr3}, in an elementary proof of the prime number theorem \cite{dpa}, in deriving anomalous scaling laws in turbulent flows of fluid and plasma \cite{dp1,dp2}, and in estimating  amplitudes and orbits of limit cycle orbits in nonlinear Rayleigh Van der Pol equations \cite{dpp}. 
The present paper and a subsequent one with applications in nonlinear differential equations (denoted II) \cite{dpnw} may be considered as an attempt of offering a rigorous analytic framework accommodating duality structure induced jump differential equations leading naturally to anomalous scaling laws those arise abundantly in the dynamics of complex systems \cite{west, tsallis}.

The paper is structured as follows. Sec. 2.1 and 2.2 give details of the analytic results of the formalism. In Sec.2.3, we  briefly collect the measure theoretic implication of the analytic structures developed in the earlier sections. Sec. 2.4 gives the basic definitions leading to consistent differential calculus on the extended real number system $\bf R^*$. In Sec. 3.1 we analyze the asymptotic influence of the intrinsic nonlinearity of $\bf R^*$ on the asymptotic states of linear and dispersive wave equations leading to pwer law attentuation. Finally, in Sec.3.2, we formulate differential calculus on a Cantor subset of $\bf R$.

{\bf Notations:}
Notations and symbols used in this paper are generally defined (explained) in the sequel.
We denote, in particular,  by $\bf Q$ the field of rational numbers, $\bf R$ the field of real numbers. Corresponding non-archimedean extension of  $\bf R$ accommodating duality structure is denoted by $\bf R^*$. Natural numbers are denoted by $n,\ N$,  real variables by $x, \ y$ etc.

Let $a=\{a_n\}, \ b=\{b_n\}$ be  sequences of rational numbers. Then, by $f(a,b)$ we mean the sequence $\{f(a_n,b_n)\}$. For example,
$\log_{a^{-1}} b/a=\{\log_{a_n^{-1}} |b_n/a_n|\}$, for $a_n>0$ and  $|.|$ being the usual norm. Big $O(\cdot)$ and small $o(\cdot)$ notations have their standard text book meanings.

\section{Formalism}
\subsection{Ordered field extension}
Recall that the real number system $\mathbf{R}$ is  constructed generally either as
the order complete field or as the metric completion
of the rational field $\mathbf{Q}$ under the Euclidean metric $|x-y|,
\ x,y\in\mathbf{Q}$. More specifically, let S be the set of all Cauchy
sequences $\{x_{n}\}$ of rational numbers $x_{n}\in\mathbf{Q}$. Then $S$ is a
ring under standard component-wise addition and multiplication of two rational
sequences. Then the real number field $\mathbf{R}$ is the quotient space $S/
S_{0}$, where the set $S_{0}$ is the set of all Cauchy sequences converging to
$0\in \bf Q$ and is a maximal ideal in the ring $S$. Alternatively, $\mathbf{R}$
can be considered as the set $[S]$ of equivalence classes, when two sequences
in $S$ are said to be equivalent if their difference belongs to $S_{0}$.

The present nonclassical, duality enhanced extension $\mathbf{R}^{*}$ of $\mathbf{R}$,
 accommodating emergent asymptotic structures,  is based on a
\emph{finer} equivalence relation that is defined in $S_{0}$ as follows: let
$a:=\{a_{n}\}\in S_{0}, \ a_n>0$ for $n>N$, $N$ sufficiently large.
This specially selected sequence $a$ is, henceforth, said to parametrize
a {scale}. We now claim that relative to the chosen scale $a$, the set $S_0$ is
{\em renormalized into a polarized set} in the form ${\tilde S}^a_{0}:=S^a_{0}\cup S^0_0$
where, $S^a_{0}=\{A^{\pm}|\ |A^{\pm}|=\{a_{n}\times a_{n}^{\mp b^{\pm}_{n}}\}\}$ when
$b^{\pm}_{n}> 0$ are  any two non-null Cauchy sequences, so that $A^\pm$ are again null and $S^0_0=
\{A^{\pm}_0| \ |A^{\pm}_0|=\{a_{n}\times a_{n}^{\mp b^{\pm}_{n}}\}\}$ when  $b^{\pm}_n$ is
null or divergent (either to $\infty$ or oscillating).  As a consequence, the exponentiated sequence
$\{b_n\}$ belongs to the set of all (Cauchy or not) sequences of rational numbers. Further, ${\tilde S}^a_0=S_0$.

Clearly, $\tilde S^a_{0}\subset S_{0}$, as sequences of $\tilde S^a_{0}$
converge to 0 in the metric $|.|$.  Conversely, given $\bar a=\{\bar a_n\} \in S_0$,
there exists $\bar A\in S^a_{0}$ where $|\bar A|=\{a_n\times a_n^{-v_n(\bar a)}\}$ for $n>N$, where
$v_n(\bar a)= { |\log_{a_n^{-1}}|{\bar a}_n/a_n| \ |}$ ,
when $\bar a_n\neq 0$ for a sufficiently large n, and $\lim v_n(\bar a) $ as
$n\rightarrow \infty$ exists. Otherwise, that is, when $\lim  v_n(\bar a) $ as
$n\rightarrow \infty$ does not exists,  we set the exponent equal to $\infty$.
As a consequence, we have the stronger equality, $\tilde S^a_{0}=S_0$ for any $a\in S_0, \ a\neq 0$, as claimed.

\begin{example} {\rm Let $a=\{ n^{-1}\}$ be chosen as  scale. Consider sequences $\{\bar a_n\}= \{n^{-\alpha}\}, \alpha>0$,
$\{\bar b_n\}= \{n^{-n}\}$, and $\{\bar c_n\}=\{(2n-1)^{-1}, (2n)^2\}$. Then, $\lim v_n(\bar a)=|\alpha-1|, \ \lim v_n(\bar b)=\infty$ and
$\lim v_n(\bar c)$ is oscillating. Moreover, $\lim v_n(a)=0$.

}

\end{example}

Next, we assume that the exponentiated sequences in $\tilde S^a_0$, viz,
$b_{n}^{\pm}$, must respect the \emph{duality structure} defined by $(b_n^{-})^{-1}\propto b_n^{+}$ for $n>M$.
The duality structure extends also over the limit elements: viz., $\mathbf{R}\ni\ (b^{-})^{-1}\propto b^{+}$
where $b^{\pm}_{n}\rightarrow b^{\pm}$ as $n\rightarrow\infty$, when one assumes $\infty^{-1}=0$.

Next, define an equivalence relation in $S^a_{0}$ declaring two sequences
$A_{1},A_{2}$ in the set $S^a_{0}$ equivalent if the associated exponentiated
sequences $b_{n}^{i}\equiv \log_{a_n^{-1}} (a_n^{-b^i_{n}}), \ i=1,2$ differ by an element of $S_{0}$
for $n>N$.  Henceforth, all the sequences considered are
interpreted as representations of equivalence classes in the quotient ring $\tilde S=S/S^0_{0}$.

The following Lemma now characterizes the general structure of the renormalized polarized set $\tilde S^a_0$.

\begin{lemma} {\bf Refined Partition of Null Sequences}
 The choice of a privileged scale (sequence) $a=\{a_{n}\}\in { S}_{0}$
introduces a polarizing effect and partitions ${ \tilde S}^a_{0}$ into equivalence
classes having the renormalized representations given by

${\tilde S}^a_{0}=\{ \{A_n^0\},\{A_{n}^{\infty
}\},\{A_{n}^{-}\}, \{A_{n}^{+}\}, \{A_{n}^{^{\prime }}\}\}$ where

$(1.a)  \ \ |A_{n}^{0 }|= a_{n}\times a_{n}^{{\alpha_n }(1+o(1))}(1+H.O.T)$

$(1.b) \ \ |A_{n}^{\infty }|= a_{n}\times a_{n}^{{\beta_n }(1+o(1))}(1+H.O.T)$

$(1.c) \ \ |A_{n}^{+}|= a_{n}\times a_{n}^{-k^n_{+}(1+o(1))}(1+H.O.T)$

$(1.d) \ \ |A _{n}^{-}|= a_{n}\times a_{n}^{k^n_{-}(1+o(1))}(1+H.O.T)$

$(1.e) \ \ |A_{n}^{^{\prime }}|= a_n\times a_{n}^{\pm \gamma^{\pm}_n(1+o(1))}(1+H.O.T)$

\noindent where $\ n\geq N, $  for sufficiently large $N $. Moreover, $|\alpha_n|\rightarrow | 0$, $|\beta_n|\rightarrow \infty$ or does not exist, $|k^n_-|\rightarrow |k_-|>1$,  $|k^n_+|\rightarrow |k_+|: \ 0< k_+<1$, and $\gamma^{\pm}_n\rightarrow \gamma:, \ 0<\gamma\leq 1$, for $n\rightarrow \infty$. This partition is consistent with the duality structure in the sense that the sequences $A_n^0, \ A_n^{\infty}$ and $A_n^+, \ A_n^-$ are respectively dually related when $A_n^{\prime}$ are self dual. Further, $S_0^0=\{\{A_n^0\}, \{A_n^{\infty}\}\}$.
\end{lemma}

The proof of this lemma follows directly from definitions introduced above.

We now define a {\em finer} asymptotic linear
ordering $\leq_a$ in the ring $\tilde S$  as follows. Denote $x^i=\{x^i _n\}\in S $. Then
$x^1\leq_a x^2$ if and only if $x^1_n\leq_a x^2_n+d_n, \ d_n>0$ for all $n$ except for a finite $n$
and $\{d_n\}\in S\setminus S_0$.
 Notice  that this ordering restricted in the sub-quotient ring $S/ S_0$ is consistent with the usual  ordering $\leq$ in $\bf R$
since the full subring $S_0$ ( the maximal ideal) itself represents the equivalence
class for 0 in that case. As a consequence the linearly ordered ${\bf R}\equiv S/S_0$ is Archimedean.

However, in the presence of a scale $a$ and the associated  duality structure,
the ordering in the subring $S_0$, designated instead as the polarized set $\tilde S^a_0$,  is further reinforced to an {\em asymptotic ordering} relative to the scale $a$ by the condition that $A_1^+\leq_a A_2^+$  (or $A_1^-\leq_a A_2^-$) if and
only if   $b^{1+}_{n}\leq b^{2+}_{n} +d$ (or equivalently, $b^{1-}_{n}\geq b^{2-}_{n}+d$)
for all $n$ except for a finite set  and a rational $d>0$. Notice
that implication for the  ordering inside the bracket is a consequence of duality.
This duality enhanced ordering also identifies $A^-_i$ with $A^+_i$ for each $i$ in the sense that  either any one of these two choices is significant. Assuming that the elements $A^-_i$ are ordered, the ordering of $A^+_i$ are determined by duality.
Finally, to complete the definition of ordering, we set $A<_a B$ for any $B\in S\setminus S_0$ and $A\in S_0$.

We remark that the asymptotic ordering inherited from the duality structure is nontrivial. In the absence of it,
the usual ordering of $\tilde S$, viz, $x^1\leq x^2$ if and only if $x^1_n\leq x^2_n+d_n$ for
$\underset{n\rightarrow \infty}\lim d_n>0 $ does not extend over to $S_0\setminus S^0_0$, since in that case $\{d_n\}$ must necessarily be
a null sequence. The duality structure gives a novel ``microscopic" window to distinguish null sequences of $S_0$ via a scale dependent {\em polarization}, so to speak, of null sequences, having the typical renormalized representations, as introduced in the above paragraphs.

The asymptotic  linear ordering now realizes the set $S$ as a non-archimedean ring because $0<_a nA <_a B$ for
any natural number $n$ and $A\in S_0, \ B\in S\setminus S_0$, since $nA$ belongs to the same equivalence class of $A$
under the asymptotic ordering.
It now follows, moreover, that the subring $S_{0}^0$ is the maximal ideal under the finer ordering $\leq_a$.
 Indeed, suppose, otherwise, that $S_{0}^0\subset J\subset S$, and $J$ is an
ideal of $S$ such that $J\setminus S_{0}^0\neq \Phi$. Let $B\in J\setminus S_{0}^0$. Then either, $B\in S_0\setminus S_{0}^0$ or
$B\in S\setminus S_0$. In the second case we are done, since multiplicative inverse $B^{-1}$ exists in $ S\setminus S_0$
so that $BB^{-1}=1\in J, \Rightarrow \ J=S$. Otherwise,  the exponentiated sequence
$|\log_a B/a|$ exists non-trivially along with its multiplicative inverse in $S\setminus S_0$, and hence, again
 $J=S$ under the duality enhanced asymptotic ordering. As a
consequence, the quotient space ${\bf R}^*_a=S/S_{0}^0$ is a non-archimedean ordered field, that contains $\bf R$ as a
proper subfield. The set ${\bf R}^*_a$ is said to be a {\em softer/dynamical } model of real numbers, relative to the
scale $a$,   compared to the unique standard {\em hard/static} real number model $\bf R$. For simplicity of notation, we shall,
however, use the symbol $\bf R^*$ to denote a softer model assuming the presence of the scale $a$ understood.

Because of the non-archimedicity of $\bf R^*$, the softer model of real numbers now enjoys host of new elements,
usually called in the literature the infinitesimally small and infinitely large numbers and are defined as follows. In the present
approach, however, such infinitely small or large numbers exists only in an  asymptotic sense that is made precise in Sec.2.2.

\begin{definition}{\rm A positive number $0<x\in \bf R^*$ is infinitely small if $x<_a r$ for all positive $r\in \bf R$.
It is infinite if $x >_a r$ for all positive $r \in \bf R$. $x$ is finite if it is not infinite. For definiteness, infinitely small and large numbers in $\bf R^*$ are called {\em nontransient} or {\em dynamic}, as compared to the conventional non-standard models where infinitesimals may be considered as non-dynamic or transient.
}
\end{definition}

As a consequence, we have

\begin{proposition} The field $\bf R^*$ is a non-archimedean ordered field extension of $\bf R$, when the usual linear ordering $\leq$ of $\bf R$ is extended to a finer asymptotic ordering $\leq_a$ that acts nontrivially in the subring of null rational sequences $S_0$.  Moreover, the finite elements of the non-archimedean extension $\bf R^*$ constitute a subring $F$ of $\bf R^*$. The dynamic infinitesimal elements of $\bf R^*$ forms the maximal ideal  $F_0$ in $F$. The residual class field $R=F/F_0$ is (order) isomorphic to $\bf R$.

\end{proposition}

\begin{proof}{\rm The construction of the finer ordering leading to the stated field extension is already furnished in the above paragraphs. The remaining statements follow from the observation that the canonical map $\phi:F\rightarrow  R$ is order preserving \cite{robin2}.
}
\end{proof}

\begin{remark}{\rm The present non-archimedean, soft extension $\bf R^*$ of $\bf R$ should be contrasted with the conventional nonstandard extensions $^*{\bf R}$ \cite{robin}. The set $\bf R^*$ can of course be interpreted as a novel nonstandard model with infinitesimal numbers. However, these arise purely (intrinsically) as a consequence of the duality structure and the associated total order $\leq_a$. Further, the construction of $\bf R^*$ is based purely on the ring of Cauchy sequences on the field of rationals $Q$, analogous to $\bf R$, but for the duality induced enhanced ordering. The non-Cauchy rational sequences that may arise as the expontiated sequences in the renormalization definitions are eliminated by simply identifying them with the null sequences via duality. As a consequence infinitely small and large numbers here may be considered more natural compared to those  in a nonstandard model $^*\bf R$, the role of an ultrafilter   is relegated to the level of rational Cauchy sequences thanks again to the duality structure. Possible relation of the present model with that of the Internal theory of Nelson \cite{nelson} will be considered separately.
}
\end{remark}

Recall that the (hard) real number system $\bf R$ is order complete and hence connected. The non-archimedean extension $\bf R^*$ is an ordered field, however, it is not order complete. As the archimedean property is violated by the asymptotic ordering relative to a scale, the connectedness of the subfield $\bf R$ is broken as the asymptotic ordering realizes new unconnected structures in the vicinity of a point $x\in \bf R^*$. We investigate the disconnected structure of the asymptotic neighbourhood of $0\in \bf R^*$, say, in the following, exploiting the asymptotic valuation that is available naturally in $\bf R^*$.

\subsection{Valuation}
The study of a non-archimedean field is  facilitated by the natural
existence of a valuation i.e. the non-archimedean absolute value. We begin by
introducing a valuation in the set of null Cauchy sequences $S_0$ of $Q$.
We show that this valuation on the null sequences is inherited by an asymptotic neighbourhood of
0 of $\bf R$. Clearly, the usual metric $|.|$, fails
to distinguish elements belonging to two distinct  finer equivalent
classes realized by the asymptotic order. However, the metric defined as the natural logarithmic extension of
the Euclidean norm, generically called the \emph{asymptotically visibility
metric} is defined by
\begin{equation}\label{valu1}
V(A_{1},A_{2}) :=\underset{n\rightarrow\infty}\lim \bigg |\log_{|a|^{-1}} |A_{1}-A_{2}|/|a|\bigg |
\end{equation}
where $a=\{a_{n}\}\in \tilde S^a_{0}$. This limit
exists because the concerned exponentiated sequences $b_{{n}}^{\pm}$ of Sec.2.1 are Cauchy or made Cauchy by the duality structure.The
sequence $a$, as stated above,  acts as a natural {\em scale} relative to which elements
in $S_{0}$ gets nontrivial values and hence become distinguishable. Alternatively,
as soon as a sequence $a\in S_0$ is declared as a scale, the zero-set $S_0$ is said
to have been {\em asymptotically polarized} as $\tilde S^a_0$, so as to assign nontrivial
values (analogous to an effective magnetization in a paramagnetic material under an external magnetic field) to otherwise indistinguishable (i.e. unpolarized) elements in the zero set $S_0$ (the equivalent class for $0\in \bf R$).
As will be shown, the singleton set $\{0\}$ of $\bf R$ would then be extended to a nontrivial bounded subset of $\bf R$
under the action of the visibility metric $v$.  Note
that the mapping $v:S_0\rightarrow\mathbf{R^{+}}$ defined by
\begin{equation}\label{vnorm1}
v(A)=\underset{n\rightarrow\infty}\lim\bigg |\log_{|a|^{-1}} |A|/|a|\bigg |
\end{equation}
 is actually a nontrivial norm \cite{dpr1,dpr2,dpr3}.  Notice that $v(a)=0$, besides the
standard definition $v(0)=\{0\}$. It follows from the definition that the visibility norm $v(A)$ actually
quantifies the rate of convergence of null sequences in $S_0$ relative to
the scale $a$.

\begin{example}{\rm
Consider two null sequences $A=\{a^n\}, \ B=\{b^n\}, \ 0<a < b<1$ and choose $A$ as the privileged scale. The  effective renormalized value of $B$ relative to the scale $A$ is given by $v(B)=\log_{a^{-1}} b/a$. Notice also that $v(kB)=v(B)$ for any finite non-zero $k\in \bf R$, since $\frac{\log k}{n\log a}$ vanishes as $n\rightarrow \infty$.}  $\Box$
\end{example}

\begin{example}{\rm
Consider null sequences $A=\{a^n\}, \ B=\{b^n\}, C=\{c^n\}, \ \ 0<c<a < b<1$ and choose $A$ as the privileged scale.  The  effective renormalized values of $B$ and $C$ are given respectively by $v(B)=\log_{a^{-1}} b/a$ and $v(C)=\log_{a^{-1}} a/c$. Clearly, $v(B)<1$, and $v(C)>1$ for $c<a^2$, so that for such a $c, \ \exists \mu>0 \ \Rightarrow$  $v(B)v(C)=\mu$. This is an example of duality structure at the level of sequences viz, in $S_0$. The null sequences $B$ and $C$ are then interpreted as visible and invisible asymptotics respectively relative to the scale defined by $A$.
}  $\Box$
\end{example}

\begin{lemma}
The mapping $v:\tilde S^a_{0}\equiv S_0\rightarrow\mathbf{R^{+}}$ is well defined and also
defines a nonarchimedean norm (absolute value). Further, $v$ is discretely valued and identifies nontrivially $A^+$ with
$A^-$ via duality where $A^{\pm}=\{a_n\times a_n^{\mp b^{\pm}_{n}}\}$.
\end{lemma}

\begin {proof}{\rm  The first part is obvious. For the second part, clearly $v(A)\geq 0$
since $b^{\pm}_{n}>0$ and $v(A)=0$ if and only if $\{b^{\pm}_{n}\}\in S^0_0$. To prove
the strong triangle inequality, we note first that $v(A^-)\propto 1/v(A^+)$, by duality,
so as to identify the elements $A^-$ and $A^+$ nontrivially. As a consequence, it suffices to consider
elements of the form $A^-_i, \ i=1,2$ such that $A^-_1<A^-_2$.
We have $v(A^-_1+A^-_2)=\underset{n\rightarrow\infty}\lim \bigg
|\log_{|a|^{-1}} |A^-_1+A^-_2|/|a|\bigg |=v(A^-_2)\leq {\rm max}\{v(A^-_1),v(A^-_2)\}$, and hence the
strong triangle inequality is true.
Next, to  prove $v(A_1A_2)=v(A_1)v(A_2)$, let $v(A_i)=\sigma^{\alpha(A_i)}, \ 0<\sigma<1$
and $\alpha(A_i)$ are the {\em exponential valuation} satisfying  $\alpha(A_1+A_2)\geq {\rm min}\{\alpha(A_1),
\alpha(A_2)\}$ and $\alpha(A_1A_2)=\alpha(A_1)+\alpha(A_2)$. With this choice of $v(A_i)$ the valuation is discretely valued, that also, in retrospect, proves the
lemma.
}\end{proof}

The extended real number system $\mathbf{R}^{*}$  admitting duality induced
\emph{fine structure}, now, is given by quotient $\mathbf{R}^{*}:=S/S^0_0$ when
convergence in $S_0$ is induced naturally by the asymptotically visibility metric
$V(x,y):=v(x-y)$. Further, The asymptotically extended field $\mathbf{R}^{*}$ inherits the duality structure
of the ring of Cauchy null sequences of $Q$.
Since the maximal ideal $S^0_{0}\subset S_0$, the extended field is non-archimedean and hence contains nontrivial
infinitely small and large asymptotic elements defined by
\begin{definition}
An arbitrarily small positive number $\delta $ in $\bf R$ such that $\delta\rightarrow 0^+$ is said to define an asymptotic
scale in the associated extended space $\bf R^* (:=R^*_{\delta}$).
Relative to this scale arbitrarily small `real' numbers $x_\pm$ satisfying
\begin{equation}\label{infsm}
0<|x_-|<\delta\leq |x_+|
\end{equation}
and the inversion law
\begin{equation}\label{Rduality}
\delta/|x_-|\propto |x_+|/\delta
\end{equation}
are called {\em visible} and {\em invisible} asymptotic (infinitesimal) numbers. For a given visible asymptotic $x_+$, the proportionality constant $\lambda$ (say) in the above inequality is independent of  $x_+$
but may depend on $\delta$ and $x_-$. The corresponding {\em visibility norms}  of the respective asymptotics are defined by
\begin{equation}\label{vnorm2}
v(x_\pm)=\underset{\delta\rightarrow 0}\lim\bigg  |\log_{\delta^{-1}}|\frac{x_\pm}{\delta}|\bigg |, \ x_\pm\neq 0
\end{equation}
and $v(0)=0$. Large asymptotic numbers are given respectively by $x_\pm^{-1}$ relative to the scale $\delta^{-1}, \ \delta\rightarrow 0^+$. The two relations (\ref{infsm}) and (\ref{Rduality}) together with the associated exponential inversion law for the visibility norms
\begin{equation}\label{dual}
v( x_-)\propto 1/v(x_+), \ v(x_-)>1, \ v(x_+)<1\tag{2.4a}
\end{equation}
define what is called {\em asymptotic duality structure} in $\bf R^*$. This relation, as it were, introduces a nontrivial gluing of  two duality related asymptotic elements $x_\pm$. The relevant proportionality constant denoted by $\mu$, so that $v(x_-)v(x_+)=\mu$, may now depend only on $x_-$ for a given $x_+$.
\end{definition}
The inversion relations (\ref{Rduality}) and (\ref{dual}) on visibility norms of associated relative asymptotic numbers are nontrivial. The duality, so defined, realizes the ability of an invisible   element to communicate with a visible one at the level of the visibility norm. Otherwise, i.e. in the absence of (\ref{dual}), say, the basic inversion  law of the form  (\ref{Rduality}) would simply identify the invisible sectors $(0,\delta)$ and $(\delta^{-1},\infty)$, for example, $x_+$ with $x_+^{-1}$ by the identity $v(x_+)=v(x_+^{-1})$, that follows from the basic (primary) {\em inversion invariance} of the visibility norm, viz, eq(\ref{infsm}) (c.f. Lemma 2.3).

The limit in (\ref{vnorm2}) exists by construction (c.f. Sec.2.1 and Lemma 2.2). Notice also that $v(\delta)=0$, that means that the scale $\delta$ actually belongs to the equivalent class of $0 \in \bf R^*$. As a consequence, the valuation (i.e. the visibility norm) realizes visible asymptotic numbers of the form $0<\delta<x, \ x\rightarrow 0$ as having nontrivial {\em effective real values} $v(x)$ by a nontrivial identification of the scale $\delta$ with 0. We denote this extended equivalence class of 0 in $\bf R^* $as $\bf O$. To emphasize, the set $\bf O$ consists of visible asymptotic numbers of $\bf R^*$. As a consequence, every real number  $x$ is extended in $\bf R^*$ as $x^*=x + {\bf O}, \ 0^*:=\bf O$.

\begin{example}{\rm 1. To give an example of duality structure in $\bf R^*$, let $x=\delta^{1-\beta}, \ 0<\beta<1$ so that $0<\delta<x$, $x\rightarrow 0$ as $\delta\rightarrow 0$. Then there exists one parameter family of invisible asymptotics $\tilde x_k= k \delta^{1+\alpha}, \ 0\neq k\in \bf R$ so that $\tilde x_k x=\lambda(\delta)$ $\Rightarrow$ $v(x)=\beta$, $v(\tilde x_k)=\alpha$ and $\alpha>1$, by duality.

2. As a consequence, the (invisible) open interval $(0,\delta)$ can be covered by countable number of open subintervals of the form $\tilde I_i=(\tilde x_{i},\tilde x_{i+1}), \ x_0=0, \ i=0,1,2,\ldots$ such that $v(\tilde I_i)=\alpha_i>1$ . Duality structure then determines a non-trivial open interval of visible asymptotics  $x_i$ having discrete number of effective values $v(x_i)=\beta_i\propto \alpha_i^{-1}$ even in the limit $\delta\rightarrow 0$.  The effective values of visible asymptotics corresponds to the discrete ultrametric absolute values in the non-archimedean extension $\bf R^*$. Hence the nontrivial neighbourhood of 0, called an asymptotic prolongation $\bf O$ of the singleton $\{0\}$, has the structure of a Cantor set $C_{\delta}$  \cite{na}. The associated  effective values of visible asymptotics $v(x)$ has  the structure of a Devil's staircase (Cantor) function (c.f. Lemma 2.4,2.5 and Theorem 2.1) corresponding to $C_{\delta}$.
}  $\Box$
\end{example}

The usual Euclidean absolute value (norm) on $\bf R$ now gets extended to the composite norm $||.||: \mathbf{R}^{*} \rightarrow {\bf R}^{+}$ given by
\begin{equation}\label{nmetric}
||x^*||=
\begin{cases}
        |x| & \text{if $x^*\in {\bf R}^*\setminus \bf O$} \\
				v(x) & \text{if $x^*\in \bf O$}
\end{cases}
\end{equation}
 The composite norm assigns the usual absolute value $|x|$ to a finite (non-asymptotic) $x^*\in {\bf R}^*\setminus \bf O$ for which $v(x)=0$, by definition. Nontrivial absolute value $v(x)$ is, however, assigned to an asymptotic number $x\in \bf O$ i.e., in other words, when $x\rightarrow 0$ in $\bf R$, (having the usual limiting value $|x|=0$). The associated metric function on $\bf R^*$ is denoted by $d(x^*,y^*)=||x-y||$.

\begin{remark}{\rm Each  number in $\bf R^*$ is assigned with the natural preassigned scale $\delta$. However, for a non-zero, {\em finite}  number $x$ of $\bf R$, the natural scale is 1.  As a consequence, the visibility norm is defined and meaningful  only for asymptotic numbers in $\bf 0$, i.e. for the {\em visible} asymptotic numbers and hence, we set, by definition, $v(x)=V(x,0)=0, \ \forall x,\  {\bf R}$. This is consistent with the fact that the direct evaluation of $v(x)$ gives the trivial value $v(x)=1$ when $|x|$ is not asymptotic and $\delta\rightarrow 0^+$. Thus the visibility norm maps the real subfield to the trivial value 1. As a result, the above definition makes sense.  Alternatively, one may modify the definition (\ref{vnorm2}): $\tilde v(x)=\underset{\delta\rightarrow 0}\lim\bigg |\log_{\delta} |x|\bigg |$ to get the value 0 for $|x|$ finite.  For a visible asymptotic $x$, one instead would get $\tilde v(x)=1-v(x), \ v(x)<1$. We, however, prefer the  the first choice in the present paper.}
 \end{remark}

Let us now collect a few basic properties of the visibility norm $v(x)$ in the following two Lemma.

\begin{lemma} {\bf Shift, Scaling and Inversion Invariance:}
Let $x$ be an asymptotic number in {\bf 0} and $v(x):=v_{\delta}(x)$ be the visibility norm of $x$ relative to the scale $\delta$. Then \\
(1). $v(kx)=v(x),$ for any constant $k$, \\
(2).  $v_{k\delta}(x)=v_{\delta}(x),$ for any constant $k>0$, \\
(3). $v(x^{-1})=v(x)$, \\
(4). $v(x+x_0)=v(x)$, for any $x_0\in \bf 0$ such that $0<\delta\leq |x_0|\leq |x|$.
\end{lemma}

\begin{proof}
The scaling and  inversion invariance properties, viz, (1), (2)  and (3) follow directly from the definition of the visibility norm (\ref{vnorm2}). For additive invariance, note that $v(x+x_0)=|\log_{\delta^{-1}}|\frac{|x||(1+x_0/x|)}{\delta}|=v(x)$, as $\delta\rightarrow 0$, since the second term $\frac{\log(1+x_0/x)}{\log \delta^{-1}}$ vanishes in that limit, the numerator $\log (1+x_0/x)$ being at most finite.

\end{proof}

{\bf Note:} The additive invariance (3) tells that the visibility norm is an ultrametric norm satisfying the strong triangle inequality, see Lemma 2.2.

\begin{lemma}
Let $x_-^i, \ i=1,2,\ldots n$ be a set of invisible asymptotics for a given visible asymptotic $x_+$. Let the associated proportionality constants be $\lambda^i(\delta)$. Then (i) $v(x_-^i)=v(x_+)=\sqrt{ \mu}$, a constant independent of $i$, if and only if $\lambda^i(\delta)$ is slowly varying such that $\lim \frac{\log \lambda^i(\delta)}{\log \delta}=0, \ \delta\rightarrow 0$. The constant $\mu$ is defined by the duality structure $v(x_+)v(x_-^i)=\mu^i$, so that $\mu^i=\mu$ for each $i$. However, (ii) for a nontrivial $\lambda^i$ satisfying $\lim \frac{\log \lambda^i(\delta)}{\log \delta}={\tilde\mu}^i, \ \delta\rightarrow 0$,   the value of the visibility norm for the visible asymptotic $x_+$ is  a  quadratic irrational given by $\frac{\mu^i}{v(x_+)}=v(x_+)+{\tilde\mu}^i$, for each $i$. In other words, a visible asymptotic may get different irrational values depending on the  type of invisible asymptotic indicated by a nontrivial proportionality constant $\lambda^i(\delta)=\delta^{{\tilde\mu}^i}(1+o(1))$.
\end{lemma}

The proof of this Lemma follows easily from relevant definitions. The slow variation condition in Case (i) is satisfied trivially in the special case when $\lambda^i(\delta)$ is a constant. More general example is given by $\lambda^i(\delta)=\delta^{\frac{a_0}{(\log\delta)^{2}}}$.

\begin{example}{\rm Let $x_+ =p\delta^{1-\alpha}, \ p>0;  \ x_-=q\delta^{1+\beta}, \ q>0$ and $0<\alpha<1$, $\beta>1$ are  constants. Suppose $ x_{-}x_+=\lambda(\delta)\delta^2$ and $\lambda(\delta)\sim O(\delta)$ is slowly varying such that $\frac{\log \lambda}{\log \delta}\rightarrow \mu$ as $\delta\rightarrow 0$. Let the duality structure is given by  $\alpha\beta=\tilde \mu, $ a constant. As a consequence,  the consistency condition is given by $\frac{\tilde \mu}{\alpha}=\alpha + \mu$. For $\mu$ and $\tilde \mu$ both close to 1, one obtains $\alpha\approx \frac{\sqrt 5-1}{2}$, the golden ratio.

}
\end{example}

\begin{definition}
The inversion relation (\ref{infsm}) with a slowly varying, almost constant $\lambda$ satisfying the limiting condition of Case (i) of  Lemma 2.4 is called an {\em asymptotic circle inversion}. Otherwise, i.e., when that of Case (ii) is valid, the same is called an {\em asymptotic elliptic inversion}.
\end{definition}

\begin{corollary}
Visible and invisible asymptotics defined by circle inversion are indistinguishable. In other words, asymptotics respecting circle inversion are self-dual.
\end{corollary}

\begin{example}{ The pair defined by the prime counting function $\Pi(n)$ and the $n$th prime $P(n)$ constitutes a nontrivial example of a self dual pair of asymptotics, because of the prime number theorem that states that $\Pi(n)P(n)/n^2\rightarrow 1$ as $n\rightarrow \infty$. Examples of trivial self-dual pairs are $(0,n)$ and $(1,n^2)$. In these examples the sequence $n$ is chosen as the scale.

}

\end{example}

More details of the circle and elliptic inversions together with respective asymptotics and their implications in number theory would be studied separately.

\begin{remark}{\rm The visibility norm  is originally considered by Robinson as exponential valuation to study asymptotic number fields $^{\rho}R$ in the nonstandard model $^*\bf R$ \cite{robin2}.  The norm (\ref{vnorm2}), however, is used independently recently \cite{dpr1,dpr2,dpr3} in the context of fractal sets and other applications \cite{dpa,dp1,dp2,dpp}.}
\end{remark}

 Now, returning back to the field extension, let us notice that the extension $\mathbf{R}^{*}$ is indeed a field since the quotient ring
$S/S^0_{0}$ is free of zero divisors. The norm $v(x)$ acts nontrivially
on an asymptotic neighbourhood of 0, and  is, in fact, an ultrametric satisfying the strong triangle inequality.
\begin{lemma}{\bf Ultrametric Property:}
The visibility norm acting on  $(0,\delta)\subset {\bf R^*}, \ \delta\rightarrow 0$   is a discretely valued  ultrametric norm.
\end{lemma}

Proof of this lemma is analogous to that of Lemma 2.2. However, for completeness, we include a proof, although the line of reasoning is almost same.

\begin{proof}
First, we prove strong triangle inequality for invisible asymptotics in $(0,\delta)$. Let $0<x_1<x_2<x_1+x_2<\delta$ such that $v(x_1)>v(x_2)$. Then as in Case 3, Lemma 2.3, we have $v(x_1+x_2)=\bigg |\log_{\delta^{-1}} \frac{\delta}{x_1+x_2}\bigg |=|\log_{\delta^{-1}} \frac{x_1(1+x_2/x_1)}{\delta}|=v(x_1)$, in the limit $\delta\rightarrow 0$, since the second term $\frac{\log(1+x_2/x_1)}{\log \delta^{-1}}$ vanishes in that limit, the numerator $\log (1+x_2/x_1)$ being at most finite. As a result, $v(x_1+x_2)\leq{\rm max}\{v(x_1),v(x_2)\}$. The relation $v(x_1x_2)=v(x_1)v(x_2)$ follows as that in Lemma 2, with  discretely valued exponential valuation $v(x_1)=\delta^{a(x_i)}$ satisfying the inequality $a(x_1+x_2)\geq{\rm min}\{a(x_1),a(x_2)\}$ and $a(x_1x_2)=a(x_1)+a(x_2)$, the value group being isomorphic to the additive group of  integers.

Next, for $y_1,\ y_2$ satisfying $0<x_1<x_2<\delta<y_2<y_1$ and $x_iy_i\propto\delta^2, \ i=1,2$ such that $v(y_1)>v(y_2)$, the above two properties follow as above from the {\em inversion invariance} of the visibility norm $v$.

Finally, $v(x)>0$ for any nontrivial asymptotic number $x$ and $v(x)=0$ implies that $x=0$, with the caveat that $0\in \bf R^*$ stands for the equivalence class of ordinary 0($\in \bf R$) together with all possible scales $\delta$.
\end{proof}

\begin{remark}{\rm The above result also follows from the observation that $v$ maps $\mathbf{R}$ to
the singleton set $\{1\}$. Further,  discrete ultrametricity of  $v$ tells that the nontrivial value set of $v$ viz., $v(S_{0})$
is countable. As a consequence, the set $S_{0}$ and hence the corresponding asymptotic neighbourhood ${\cal N}_a(0)$
of $0\in \bf R^*$ (c.f. Lemma 2.4), viz, $\bf O$, is totally disconnected and perfect in the induced topology \cite{katok}. It thus follows that $\bf O$ is homeomorphic to a  Cantor set, having Hausdorff dimension   $0<{\rm dim}_H C<1$. By a Cantor set $C$ we mean the equivalence class of totally disconnected, compact and perfect  sets $\{\tilde C\}$ having identical Hausdorff dimension. Although, the said homeomorphism hold for a general class of Cantor subsets of $\bf R$, we, however, tacitly assume, in the following, only a self-similar Cantor set. It should, however, be remembered that the duality enhanced  asymptotic neighbourhood $\bf O$ of $\bf R^*$ can , in fact, be much more complicated than a simple Cantor set having a uniform Hausdorff dimension. In general, $\bf O$ may look like a multifractal set corresponding to a multifractal measure, rather than simply a fractal measure induced by a Cantor function. Such general constructions would be considered separately.
}
\end{remark}

\begin{lemma}{\bf Existence of duality structure:}
The duality structure in $\bf R^*$ exists. As a consequence, the linear neighbourhood $(-\delta,\delta)\subset \bf R$ of 0 gets extended in the limit $\delta\rightarrow 0$ to a totally disconnected, compact and perfect, ultrametric neighbourhood ${\cal N}_a(0):=\bf O$ in $\bf R^*$.
\end{lemma}

\begin{proof}
Given $0<\delta<1, \ \exists \ x_\pm\in {\bf R}$ satisfying $0<x_-<\delta<x_+$ and $x_-x_+=\lambda\delta^2$, for a given $\lambda$ depending on $x_-$ and $\delta$. The nontrivial definition of duality in $\bf R^*$, however, realizes the above relation, instead, for all $x_-\in I\subset (0,\delta)$, an open interval, for a given $\lambda$ depending only on $x_+$, as $\delta\rightarrow 0$. Indeed, because of discrete ultrametricity (Lemma 2.5),  there exists a visibility norm $v$ such that $ v(x_-)=\alpha_i\geq 1, \forall x_- \in I_-^i$, a constant, where $\cup  I_-^i$ is a countable, disjoint cover of $(0,\delta)$ by clopen balls (in the ultrametric norm). Next, in the limit $\delta\rightarrow 0$, the open interval $(0,\delta)\subset \bf R$, conventionally, goes over to $\Phi$, the null set. However, in $\bf R^*$, the set of asymptotic values $\alpha_i$ is {\em nontrivial}. Consequently,  one envisages an asymptotic (right-) neighbourhood ${\cal N}_a^+(0)$ (say) of 0, that can be covered by another family of disjoint clopen balls $\cup  I_+^i$. Since $\alpha_i$ is real, {\em there exists} $\beta_i\propto 1/\alpha_i\leq 1$ such that $v(x_+)=\beta_i\leq 1, \ \forall x_+\in I_+^i\subset {\cal N}_a^+(0)$, a constant. As a consequence, $v(x)$ now acts as a discrete ultrametric norm on ${\cal N}_a^+(0)$ such that $0\leq v(x)\leq 1$. The singleton $\{0\}\subset \bf R$ is accordingly extended asymptotically to a neighbourhood ${\cal N}_a(0)={\cal N}_a^-(0)\cup{\cal N}_a^+(0)$, that is totally disconnected  in the induced (discrete) ultrametric topology. This proves the major part of the lemma.

The compactness follows from the observation that the value set of $v:{\cal N}_a(0)\rightarrow \bf R^+$ is a subset of $[0,1]$. Thus ${\cal N}_a(0)$ can be covered by a single clopen ball. Further, ${\cal N}_a(0)={\rm cl} ({\cal N}_a(0))$, the closure. Hence, it is also perfect.
\end{proof}

\begin{lemma}{\bf Devil's Staircase:}
The valuation $v$ on ${\cal N}_a(0)$ is a Cantor staircase function.
\end{lemma}

\begin{proof}
First, we show that the valuation $v$ that is constructed in above Lemma 6 is a non-decreasing, continuous map from ${\cal N}_a(0)$ to $I= [0,1]$.

Since ${\cal N}_a(0)$ is homeomorphic to a Cantor subset $C$ of $I$, the countable open cover  $\cup J_i$, considered in  Lemma 2.6 (with an altered notation), has finite subcover $\cup_{j=1}^{m_i} J_{ij}, \ i=1,2, ...$, for each fixed $i$,  on each of which the valuation $v$ has constant value $v(x)= \beta_{ij}$ that can be rewritten as $\beta_{ij}:=\gamma_{ij}\sigma^{i}, \ \forall x\in J_{ij}$ for $0<\sigma<1$ so that $0<\beta_{ij}< 1$ on each of the gaps $1\leq j\leq m_i$. Here, $\gamma_{ij}$ denotes a set of positive integers in ascending order (respecting of course the above constraint, since $\sigma^i\rightarrow 0, \ i\rightarrow \infty$).   The open interval $J_{ij}$ is a gap of $C$ in a standard $i$th level iterative definition of the self similar set, so that  $[\sigma^{-1}]$ (the greatest integer function) corresponds to the number of smaller closed intervals created in each iteration. For a homogeneous set $C$ with  similarity ratio $0<r<1/2$ and production of one gap at each step, $\sigma^{-1}=2=r^{-s}$ so that  $s=\log_{r^{-1}} 2$ gives the Hausdorff dimension of $C$.

A nonconstant value of $v$ is associated with a point $x$ that belongs to any of the $i$th level closed sets of $I-\cup_{j=1}^{m_i} J_{ij}$. Let $x\in [x_{k_i},x_{{k_i}+1}]$ be such an $i$th level closed set with property that $v(x_{k_i})\leq v(x)\leq v(x_{{k_i}+1}) $. At the $(i+1)$th level of iteration, a possible realization of the above inequality would be of the form $v(x_{k_{i+1}})=v(x_{k_i})\leq v(x)\leq v(x_{{k_{i+1}}+1})< v(x_{{k_i}+1}) $. Other alternative choices could arise depending on the number of gaps created by each application of the iteration process. It follows therefore that the sequence $v(x_{{k_i}})$ is increasing and $v(x_{{k_i}+1})$ is decreasing. Thus, $\lim v(x_{k_i})$, as $i\rightarrow \infty$, exists, $\underset{i\rightarrow \infty} \lim v(x_{k_i})=v(x)$ and the limit function is continuous, proving the claim.

It follows also that $\frac{dv(x)}{dx}=0, \ \forall x$ belonging to gaps of $C$. For $x\in C$, derivative, however, is not defined (that can be verified easily following standard approach). Hence the lemma.
\end{proof}

\begin{definition}
The asymptotic extension of the singleton set $\{0\}$ into a Cantor set like totally disconnected set $\bf O \ (={\cal N}_a(0))$ and the associated emergence of the Cantor's staircase function as the corresponding effective renormalized variable, as a consequence of the duality structure,  is called an {\bf asymptotic prolongation} of the singleton set $\{0\}$.
\end{definition}

\begin{definition}
A point $a^*\in \bf R^*$ is the equivalence class of duality enhanced $a=\{a_n\}\mapsto A=\{A_n\}, \ a_n\in Q$ Cauchy sequences converging  to the corresponding element $a\in \bf R$ so that $a^*=a+\bf O$. A sequence $A=\{A_n\}$ in $\bf R^*$ is said to have acquired a visibility norm $v(A)$ that equals the corresponding visibility distance $v(A-a)$ of the sequence $A$ from $a$. The set of values $v(A)$ as $A$ varies in the equivalence class is identical with the Cantor function $v(\bf 0)$ because of translation invariance.
\end{definition}

\begin{example} {\bf Bounded sequences:}
{\rm Consider the bounded sequence $A=\{2^{-n}, 3^{-1}\pm 2^{-n}\}$ with limit points $\{0,3^{-1}\}$ in $\bf R$.  In $\bf R^*$, this sequence  may, however,  be treated as  convergent relative to the privileged  scale $\delta^n, \ 0<\delta<1$ in the concerned extension provided one assigns the visibility norm to the sequence $A$ the value $v(A)=1-\frac{\log 2}{\log \delta^{-1}}$, uniformly for each individual subsequence. The value quantifies the uniform rate of convergence of a subsequence to the corresponding limit point and actually correspond to the visibility distance of each subsequences from the corresponding limit point. For a sequence of the form $B=\{2^{-n}, 3^{-1}\pm 3^{-n}\}$ with varying rates of convergence, one assigns instead $v(B)=\max \{v(B_1), v(B_2)\}=\max \{1-\frac{\log 2}{\log \delta^{-1}}, 1-\frac{\log 3}{\log \delta^{-1}}\}$, which is consistent with the strong triangle inequality when we rewrite $B$ as  $B=B_1+B_2$ where $B_1=\{2^{-n}, 0_n\}$ and $B_2=\{0_n, 3^{-1}\pm 3^{-n}\}$, when we use the notation $\{0_n\}$ to denote the zero sequence $\{0_n\}=\{ 0,0,\ldots \}$. Notice incidentally that $v(B_i), \ i=1,2$ attains the respective finite non-null value thanks again to the above stronger inequality.}
\end{example}

This example highlights the intended uniformization and  smoothing effects of the ultrametric visibility norm that essentially eliminates classical discontinuities and singularities of ordinary analysis. The following proposition summarizes this higher level smoothness enjoyed by the extended space $\bf R^*$ associated with the scale $\delta^n, \ 0<\delta<1$.

\begin{proposition}
Consider a bounded sequence $a=\{a_n\}$ in $\bf R$ with $\liminf a_n\neq \limsup a_n$. Let an associated duality enhanced convergent subsequence in the  ultrametric visibility norm in $\bf R^*$ be $A_i=\{A_{n_i}\}$. Then $\sup_i v(A_i)$ exists and is finite, where sup is evaluated over all possible convergent subsequences. Further, $v(A)=\sup_i \{v(A_i)\}$ where $A=\{A_i\}$ is the associated duality enhanced $a$.
\end{proposition}

The proof follows easily from the fact that, in a discrete ultrametric space, the convergent subsequences are at most countable and each of the corresponding duality enhanced subsequences enjoys a finite  visibility norm$(<1)$. Further, a convergent subsequence $A_i$, in fact, defines an equivalence class ($A_{i_j} \sim  A_{i_{j^\prime}}, \ \Leftrightarrow \ v(A_{i_j})=v(A_{i_{j^\prime}})$). Next, by appropriate reordering one can write $A=\sum A_i$ so that $v(A)=\sup \{v(A_i)\}$, by strong triangle inequality. $\Box$

\begin{definition}
A bounded divergent sequence $a=\{a_n\}$ in $\bf R$ is said to acquire  a uniform visibility norm defined by $v(A)=\sup_i v(A_i)$ in $\bf R^*$.
\end{definition}

Applications of this definition exposing the higher level smoothness that could be achieved in the present approach is considered in Sec.2.3 and Sec. 3.

\begin{definition}
 The set $O_v=\{x\in {\bf R^*}|v(x)\leq 1\}$ is the valuation ring. The valuation ideal is given by $I_v=\{x\in {\bf R^*}|v(x)< 1\}$. The residual class field of valuation is  given by quotient $\tilde R=O_v/I_v$.
\end{definition}

\begin{lemma}{\rm Both the residual class fields of valuation $\tilde R$ and asymptotic order $R$ (c.f. Proposition 2.1) are identical with $\bf R$.
}
\end{lemma}

\begin{proof}{\rm  Clearly, the valuation ring $O_v$ coincides with the ring of finite elements $F$. So is for the valuation ideal $I_v$ and the set of infinitesimals $F_0$. Further, the canonical injection $\phi$ maps $F_0$ into the singleton $\{0\}$.
}
\end{proof}

 It  follows, as a consequence, that $d(x,y)=|x-y|$ for any $x,\ y\in {\bf R}, \ |x-y| $  not asymptotic and
$d(x,y)=v(x-y)$ for $x,\ y\in\mathbf{R}^{*}-\bf R $. Since every Cauchy sequence in $\bf R^*$ has a limit
either in $|.|$ or in the visibility norm $v(.)$, $(\mathbf{R}
^{*}, ||.||)$ is a complete metric space. Combining Lemma 2.5, 2.6 and 2.7 together with the observation that the order incompleteness (in the sense of Dedekind i.e. having the least upper bound property) of $\bf R^*$ is due to the fact that any complete ordered field is archimedean (and isomorphic to $\bf R$), we thus have

\begin{theorem}  Equipped with the duality structure and  the associated asymptotic visibility norm the hard real number system $\bf R$ is extended to a Cauchy complete , but order incomplete (in the sense of Dedekind), non-archimedean field $\bf R^*$ relative to a privileged  asymptotic scale. The ordinary real number set $\bf R$ is realized as a subfield. The extended set $\bf R^*$ is isomorphic to $\bf R$ at the scale 1. However, at the level of the asymptotic scale, $\bf R^*$ is totally disconnected. The asymptotic prolongation $\bf O$ of 0 is homeomorphic to a Cantor subset of $\bf R$. The absolute value valuation $v: {\bf O}\rightarrow [0,1]$ is a Cantor's staircase function.
\end{theorem}

The question that arises immediately at this stage is: what is the size of the extended asymptotic neighbourhood ${\cal N}_a$(0) i.e., $\bf O$? As already remarked, the nonarchimedean extension $\bf R^*$ of $\bf R$ makes room for asymptotic extension of a point $r\in \bf R$ to a measure zero set ${\bf r}\in {\bf R^*}: \ {\bf r}=r+{\bf O}$ where ${\bf O}$ is the infinitesimal, measure zero Cantor set that is induced by the duality structure and the associated asymptotic visibility norm. As a consequence, the usual linear neighbourhood $(-\delta,\delta)$ of $0\in \bf R$ gets extended asymptotically as $(-\delta^{1-v(x)}, \delta^{1-v(x)})$ where $x$ is a visible  asymptotic element $0<\delta<|x|<<1, \  x\rightarrow 0$. The following theorem gives the precise characterization.
\begin{theorem}
The duality structure supported asymptotic extension of the linear (primary/ first generation) neighbourhood $[-\delta,\delta]$ of $0\in \bf R$ to a measure zero Cantor subset $\bf O\subset R^*$  is realized  at the limit $\delta\rightarrow 0$, so that the initial interval $[-\delta,\delta]$, as a subset of $\bf R$, reduces to the singleton $\{0\}$. An associated relative invisible element $\tilde x\in [-\delta,\delta]$ is transformed, by duality transformation, into a scale invariant visible element with a scale invariant renormalized (effective) value $X={ \xi}^{-\beta(\xi^{-1})} \ (\propto 1/v(\tilde x)) $  where $  \beta(0)=0, \ \beta(\tilde\delta)=1, \ \xi>1$, which is  a slowly varying Cantor staircase function on the asymptotically realized  second generation interval $[0,\tilde\delta], \ 0<\delta<\tilde \delta<1 $. The  new rescaled variable $ \xi^{-1}$ is an asymptotic variable leaving in the secondary interval $[0,\tilde\delta]$, where 
$\tilde\delta\approx\delta\log\delta^{-1}$.
\end{theorem}

We begin by first characterizing the structure of an asymptotic element.

\begin{definition}
Given a pair of relative visible and invisible asymptotic elements $0<\tilde x<\delta<x$, the duality structure $\tilde x/\delta\propto \delta/x$ and $v(\tilde x)\propto 1/v(x)$ induces a scale invariant renormalization group transformation of the relative invisible infinitesimal $\tilde x$ to the scale invariant visible element $X$ by $X=v(x)\log \delta^{-1}(1+o(1))=\log x/\delta \ (1+o(1))$ (c.f. Remark 2.6).
\end{definition}
The transformation concerned is a renormalization transformation since this awards vanishing small elements finite, nonvanishing values \cite{baren,golden}. Such transformations obviously form a group, called the {\em asymptotic renormalization group}, that acts nontrivially on the duality enhanced neighbourhood $\bf O$ of 0.

\begin{definition}
Let $0<\delta<\tilde \delta(\delta)<1$. The asymptotic scale defined by $\delta$ is called the primary or the first generation scale. The scale defined by any other $\tilde\delta$ respecting the above inequality is then a secondary or a higher generation scale in the sense that the scale invariant renormalized quantities such as $X$ can not be an explicit function of the first generation scale $\delta$, but may depend implicitly through higher generation scales of the form $\tilde\delta$.

Further, by the {\em asymptotic limits} of the scaling variable $\xi \  (=x/\delta)$ one means limits either of the forms $\xi\rightarrow \infty$ (i.e. $\delta\rightarrow 0$, $x$ fixed) or $\xi\rightarrow 0$ (i.e. $x\rightarrow 0$, $\delta$ fixed). On the other hand, {\em the scaling limit} $\xi\rightarrow 1$ is the so-called singular limit $\delta\rightarrow 0, \ x\rightarrow 0$ so that $x/\delta$ is finite.
\end{definition}

As it will transpire,  the scale of the form $\delta\log\delta^{-1}$ is a second generation scale relative to the first generation scale $\delta$.

\begin{lemma}
1. The invisible asymptotic element $\tilde x: \  0<\tilde x<\delta<x, \ x \rightarrow 0$ is mapped, by duality induced renormalization group transformation, to a smooth, infinitely large scaling function $\tilde x(x,\delta)/\delta \mapsto \tilde X=\tilde X(\tilde\xi)>1$ of the scaling variable $\tilde\xi=\xi^{\alpha(\xi)}, \ \xi=\delta^{-1} x>1$, where $\alpha(\xi)= \sum_m^N \alpha_i|\log\xi|^{ki}, \ \alpha_i\geq 0, \ k\geq 1$, is a  truncated Taylor expansion in the integral powers of $\log \xi$, $m=0, 1,2,\ldots$, in the asymptotic limits $\xi\rightarrow 0 \ {\rm or} \ \infty $. Further, $\alpha_0>1$.

2. For each such $\tilde x$, the scale dependent visible element $x$  then has the limiting scale invariant renormalized group representation  $x/\delta\mapsto  X= \log x/\delta \ (1+o(1))=\xi^{-\beta(\xi^{-1})}(1+o(1))$ where $\beta(\xi^{-1})$ has the leading order expansion $\beta(\xi^{-1})=\frac{\beta_0}{|\log\xi|^{b_0}}(1-\frac{\beta_1}{|\log\xi|^{b_1}}(1+o(1)))$ where $b_0$ and $b_1$ are either 0 and $1/N$ or $1/N$ and $1/(N-1)$ respectively, and $\beta_0<1$. In the limit $\xi\rightarrow 1$, however, both $\alpha(\xi)=\beta(\xi^{-1})\rightarrow 1$. The scale invariant quantities respects the duality $X\propto {\tilde X}^{-1}.$
\end{lemma}

\begin{proof}
In ordinary  calculus on real line $\bf R$, an element of the form $0<\tilde x<\delta$, for a given $\delta$  could either be a smooth function of the form $\tilde x(\delta)=\sum_m^{\infty} a_i\delta^i, \ m>0 $ or at most of the form $\tilde x(\delta)=\sum_m^{\infty} a_i\delta^ki, \ m>0, \ k>1 $ (the point $\tilde x$ then is the equivalence class of the convergent power series either of $\delta$ or $\delta^k$). In the presence of duality structure, the intervals $(0,\delta)$ and $(\delta^{-1},\infty)$ are identified by the inversion invariant visibility norm: $\tilde x/\delta=\delta^{v(\tilde x)}\mapsto  \delta/\tilde x=\delta^{-v(\tilde x^{-1})}$, so that $v(\tilde x)=v(\tilde x^{-1})$ (c.f. remarks following Definition 2.2). As a consequence, we now have, instead, $\tilde x(x,\delta)\mapsto\delta\times e^{\tilde X (\tilde \xi)}$ that is extended naturally as a smooth function of the power law scaling variable $\tilde \xi$ with above stated properties. Notice that, even as $\tilde x$ is only a relative infinitesimal depending on the scale $\delta$, the duality induced  scale invariant quantity $\underset{\delta\rightarrow 0}\lim \ \log (\tilde x/\delta)^{-1}$, that exists nontrivially, must be a  function of $|\log\xi|$ of the scale invariant $\xi>1$ i.e $\underset{\delta\rightarrow 0}\lim \ \log (\tilde x/\delta)^{-1}=\tilde X={\tilde X(\tilde \xi)}>1$.   By duality, the exponent $\alpha(\xi)>1$ (Definition 2.2) and hence the limiting scaling function is also $\tilde X (\tilde \xi)>1$, so that the original infinitesimally small scale dependent asymptotic element $\tilde x<<1$ is realized as a scale invariant infinitely large, invisible element $\tilde X (\tilde \xi)>>1$ by inversion induced renormalization transformation $\tilde x/\delta\mapsto \log (\tilde x/\delta)^{-1}=\sum_N^m a_i\xi^{i\alpha(\xi)}$ in the limit $\delta\rightarrow 0$, in the logarithmic scale.

The representation for the asymptotically realized visible element $x>0$, on the other hand, is such as to respect the duality structure at least at the leading order, so that the original scale dependent visible element $0<\delta<x<<1$ goes over to the asymptotically renormalized quantity, in association with the inversion in the exponentiated quantities in the duality structure, of the form $x/\delta\mapsto X=\log x/\delta=\xi^{-\beta(\xi^{-1})}(1+o(1))$ since $\beta(\xi^{-1})<1, \ \xi^{-1}<1$, $\delta\rightarrow 0$. Further, the new scale invariant quantities are such as to satisfy the duality $X\propto {\tilde X}^{-1}$. The power law $\xi^{-\beta(\xi^{-1})}$ for an almost constant exponential valuation  $\beta(\xi^{-1})$ is consistent with the fact that the visibility norm is a discretely valued ultrametric
\end{proof}

To emphasize, the scaling exponent $\alpha ( \xi)$ could only be a power (or a generalized)  series  in $\log{\xi}$, instead of $\delta$ or $\xi$, as the effect of said duality should be felt only either of the asymptotic or singular limits (Def. 2.9) already evaluated, but $\log\xi$ being still {\em finite and smooth}, being  away from any singularity. In fact, the logarithmic singularities at $\xi=0$ or $\infty$ are essentially eliminated (smoothed) in the context of duality structure.  The series expansion for $\alpha(\xi)$ (in the light of  Def.2 8)  is such as to make the (duality enhanced) vanishing of invisible elements $\tilde x$ much faster.

In the asymptotic limit $\delta\rightarrow 0$, the visibility (renormalized) effective value of the visible element $x$ now has the explicit, scale  invariant representation $v(x)\log \delta^{-1}=\xi^{-\beta(\xi^{-1})}(1+o(1))<0, \ \xi>1$, so that for asymptotic but,nevertheless, finite fixed values of $x$ and $\delta$, one gets $x/\delta=\delta^{-v(x)}(1+o(1))=e^{\xi^{-\beta^{(\xi^{-1})}}}(1+o(1))$. From Lemma 2.4 \& 2.5, $v(x)$ has constant value $v(x)=e^{-\beta_0}$ (i.e, $\beta(\xi^{-1})=-\frac{\beta_0}{\log\xi}$) on the closure $I_i$ of an open (right-)cover of $\bf O^+$. Variability of $v(x)$ is noticeable at a point $x_0$ that is an interior point of a closed interval of the form $\delta^\prime [x_1,x_2]$ that is immediate adjacent to $I_i$ and $\delta^\prime$ the scale associated with the iterative definition of the Cantor set representing $\bf O^+$. Define $\xi^\prime=|x-x_0|/\delta^\prime$.  In the asymptotic limit $\xi^\prime\rightarrow 0  $ (or, equivalently, in $\xi^\prime\rightarrow \infty  $ ), the constant value $e^{-\beta_0}$ of $v(x)$ would experience a slow variation in inverse powers of $|\log\xi^\prime|$ as represented in $\beta(\xi^{-1})$ in the above lemma. This inverse power variation is inherited from the duality structure and is consistent with the concomitant variation in the invisible sector. The constant value $\beta_0$, then, is attained by $\beta(\xi^{-1})$ when $b_0\rightarrow 0$, and  $ \ \beta_1=0$. The small allowed value of the exponent $b_0$ now tells that the suggested variation of $v(x)$ indeed is very slow i.e. of the form $v(x)\sim \frac{\beta_0}{|\log\log\xi|}(1+o(1))$ since $|\log\xi|^{b_0}\sim |\log\log\xi|$ for $b_0\approx 0$.
On the other hand, in the singular asymptotic limit $\xi\rightarrow 1$, the boundary value $v(x)\rightarrow 1$    is attained if $b_0\rightarrow 0, \ \beta_0=1$,  and $ \ \beta_1=0$ (c.f.Def. 2.8).    We thus have

\begin{lemma}
The staircase function $v(x)$ varies slowly with the asymptotic scaling limits $\xi\rightarrow 0 \ {\rm or} \ \infty $ in the neighbourhood of a constant value. More precisely, $v(x)$ has the asymptotic expansion
\begin{equation}\label{valuasymp}
\log (v(x)\log \delta^{-1})= \frac{\beta_0}{|\log\xi|^{b_0}}(1-\frac{\beta_1}{|\log\xi|^{b_1}}(1+o(1)))\log\xi, \ \beta_0\le 1, \ 0\le b_0<b_1<1
\end{equation}
in the vicinity of either of a point of scale variation, or as $x\rightarrow 0$ for fixed $\delta$ and $0<\xi^{-1}<1$. For $0<b_0<<1$, one, however, gets $\log v(x)= \frac{\beta_0\log \xi}{|\log\log\xi|}$
\end{lemma}

\begin{proof}( Theorem 2.2:)
The set $\bf O$ is of measure zero since it contains no line segments of length $\delta^n, \ n$, a natural number, in the limit $\delta\rightarrow 0$. By Theorem 2.1, it is a Cantor subset of $\bf R$. To prove the remaining part, let us consider $0<\delta<x<<1$ such that $x$ is visible asymptotic and $\delta$ is the associated primary scale. $x$ thus has the representation $\delta\times \delta^{-v(x)}(1+o(1))=\delta\times e^{\xi^{-\beta(\xi^{-1})}}(1+o(1))$. The renormalized effective variable thus has the form $X=\underset{\delta\rightarrow 0}{\lim}\log x/\delta= \xi^{-\beta(\xi^{-1})}$.  By Theorem 2.1, the valuation $v(x)$ is a Cantor function such that $v(\delta)=0, \ v(1)=1$. The scaling exponent $\beta(\xi^{-1})$,  in the asymptotic limiting variable  $ \xi^{-1}\in [0,\tilde \delta], \  0<\delta<\tilde\delta<1 $  also has the property of a Cantor function that is naturally inherited from that of $v( x)$, when $ x$ is assumed to live in $[\delta, \delta\times \delta^{-\delta}]\subset [\delta,1]$, so that both $\xi$ and $X$ are asymptotic. Consequently, $\inf X=0$ and $\sup X=\tilde\delta$ as $\beta(\xi^{-1})$ varies from 0 to 1 for $\xi^{-1}$ in $[0,\tilde\delta]$.  Noting that $v(\delta\times \delta^{-\delta})\approx {\frac{\log \delta^{-\delta}}{\log\delta^{-1}} }=\delta$, one must than have $\tilde\delta=\delta\log \delta^{-1}$.

As a consequence, the nontrivial secondary  asymptotic interval $[0,\tilde\delta]$ as the range space of $X$ here is  realized in the asymptotic limit $\delta\rightarrow 0$, as $x$ happens to vary slowly in the primary right neighbourhood of $\delta$ of approximate size $\tilde\delta$ of the form $x\in \delta(1,1+\tilde\delta)$, so that $\log (x/\delta)$ i.e. $ X$ belongs to the approximate interval $(0,\tilde\delta)$. The validity of the above constraints  require that the variations $\beta_0(\xi)\rightarrow 0$, $\beta_0(\xi)\rightarrow 1$ etc.  must be very slow compared to the concomitant variations of $\delta\rightarrow 0$ and $x\rightarrow 0$.  Lemma 8 estimates the exact behaviour of such variations, that are supported precisely by the duality structure. As a result, the effective renormalized variable $X$  exists with nontrivial variability in asymptotically extended (secondary) set ${\bf O^+ }\subset [0, \tilde\delta]$ and itself is a Cantor function, so that $X$, and hence, $\beta(\xi^{-1})\log \xi$, is almost constant in a gap of the said Cantor set. Notice that even as $\delta$ is assumed to be vanishing, all the associated variables viz., $X,\ \log x$ and $\log \delta$ exists and nonvanishing in the present asymptotic sense, proving the theorem.
\end{proof}

\begin{remark}
{\rm 1. The effective asymptotic variable $X=\xi^{-\beta(\xi)}$ is a renormalized quantity when the primary limit $\delta\rightarrow 0$ is evaluated even as the (visible) real variable $x: 0<\delta<x$ continues to approach 0 at a slower rate because of a nontrivial influence of duality structure. In the absence of duality, one expects, in the framework of ordinary calculus, that the above concomitant limit to yield a constant value $X=k$ when both $x, \ \delta\rightarrow 0$ linearly. In the present approach, one, however, realizes a nontrivial scaling function $X$ in the {\em new, dynamically generated } rescaling symmetric scaling variable $\xi$. The effective variable leaving in the secondary asymptotic interval $X\in [0,\tilde\delta]$ has the important property of having intermittent  behaviour, viz,   it remains almost constant over  linear segments of $[0,\tilde\delta]$ with significant O(1) slow variations only over a measure zero Cantor subset. In the subsequent section, we examine the differentiable structure of the asymptotically realized extended sets and offer reinterpretations of the above observations in the context of  linear differential equations.

2. It is pointed out in Remark 2.5 that extended asymptotic set $\bf O$ could be much more complicated multifractal set. Theorem 2.2 gives an idea how this is actually realized.  Initiated by a  privileged scale $\delta$, the singleton set $\{0\}$ gets  prolonged to the first generation Cantor set $C_{\delta}$ that arises as a subset of a second generation asymptotic interval $[0,\tilde\delta]$ equipped with an effective renormalized variable $X_{\delta}$. The asymptotic neighbourhood of $0\in C_{\delta}$ then experience a further prolongation into the second generation Cantor set $C_{\tilde\delta}$ corresponding to the second generation scale $\tilde\delta=\tilde\delta(\delta)$ and so on to ad infinitum. The original singleton $\{0\}$, as a consequence, is prolongated into a multifractal set ${\bf O}_{\delta}$ equipped with a limiting multifractal measure $X(\delta)$ extending the first generation Cantor's singular measure $X_{\delta}$. More details of this construction would be considered elsewhere (see e.g. \cite{dpnw}).

3. The  renormalization group formalism essentially eliminates trivial or divergent quantities in favour of finite observable values \cite{baren,golden}. The present asymptotic analysis exactly realizes this transformation even in the context of ordinary real analysis. We call such duality induced renormalization transformation as the {\em asymptotic renormalization group transformation}. The new finite, scale invariant asymptotic quantities would offer novel insights and reinterpretations  of divergent phenomena and processes in mathematics and its various applications. Here, we treat some simple applications of linear differential equations and  differentiability on fractal sets.  }
\end{remark}

\subsection{Measure}

In ordinary analysis the Lebesgue measure of the interval $(0,x)$ is $l((0,x))=x$ that vanishes in the limit $x\rightarrow 0$. The behaviour of cardinality measure is not that trivial: $\#((0,x))=c$, the continuum, for all nonzero $x$, but in the said limit, the cardinality vanishes suddenly, $\#((0,x))\rightarrow \#(\Phi)=0$. One may raise a pertinent question here about the ultimate fate of that continuum of elements (i.e. real numbers) all evaporating, as it were, into nothingness! It might appear reasonable, however, to imagine that the points interior to   that line segment get so  compressed under  spontaneous contraction force developed spontaneously as $x$ continues to diminish that the initial line segment is ultimately transformed into a highly irregular continuous structure (curve) that remains glued to the point 0, though along (an invisible) vertical direction. The Lebesgue measure of that continually irregular contracted set would then grow indefinitely, that was encoded in the asymptotic divergence in the scale invariant invariant quantity $\tilde X$ (Lemma 2.7). However, the duality structure now eliminates this apparent divergence and instead realizes spontaneous formation of a novel zero Lebesgue measure, but nevertheless, a finite Hausdorff measure, represented by the effective scale invariant quantity $X=\xi^{-\beta(\xi^{-1})}$, of a totally disconnected subset  $[0,\xi^{-1}]\cap{\bf O}, \ \xi^{-1}<1$. This gives a measure theoretic justification of the terminology introduced here, viz, the set $\bf R^*$ is a soft (fluid) model of real numbers when $\bf R$ being the hard (or stiff string) model. The pathological behaviour of the limiting cardinality measure is also avoided, the continuity of the measure remains intact even asymptotically in the soft model.

The above spontaneous contraction argument has an alternative explanation, that appears to be valid even in the framework of the hard model.
The equality $x=x$ could be violated asymptotically by introducing a spontaneous scaling of the  form $x\mapsto x\times a(x)^{\log x^{-1}}, \ x$ asymptotic for a small scaling factor $a(x)<1$. This reduction of measure, however, could be halted by invoking a sufficiently large factor $p(x)^{\log \tilde X}, \ a<p(x)^{-1}<1$ for a growing scaling function $\tilde X>1$, so that the linear measure is preserved asymptotically $x=x\times a(x)^{\log x^{-1}} \times p(x)^{\log \tilde X}$. This conservation principle, however, yields a diverging measure $\tilde X=x^{-s(x)}$
where $s(x)=\frac{\log a^{-1}}{\log p}>1$ even as $X$ vanishes in the limit $x\rightarrow 0$, that remains hidden in the conventional treatment of real analysis. In the present extended formalism, a finite non-linear measure is, finally, retrieved asymptotically  by duality: $\tilde X\mapsto \tilde X^{-1}\propto X \propto x^{s^{-1}(x)}, \ s^{-1}<1$ with support $\bf O$ (c.f. Theorem 2.1).

Let us now recall  the following standard results from  the theory of Hausdorff measure \cite{roger}:

(a) A Cantor function $F(x)$, being monotonic, non-decreasing and $F(0)=0$ defines a unique Borel $\sigma-$finite regular measure $\mu_F$ given by  $\mu_F(A)=\mu^h(A), \ \mu^h(A)=H^{s_c}(A)$ for any Borel set $A\subset C\subset [0,1]$ where $s_c$ being the Hausdorff dimension of the associated Cantor set $C$  and $H^{s_c}(A)$ the corresponding finite Hausdorff measure. More importantly, $F(x)=\mu_F([0,x])$.

(b) Recall also that a Cantor function $F(x)$  is continuous but not an absolutely continuous function. The associated induced measure $\mu_F$ is therefore singular, not being absolutely continuous relative to the linear Lebesgue measure. However, $\mu_F$ is, nevertheless, trivially absolutely  continuous with respect to itself and hence relative to the Hausdorff $s_c$-measure $\mu^h$.

(c) As a consequence, an application of the Radon-Nikodym theorem on differentiability of measures tells that $\frac{d\mu_F}{d\mu^h}=1$ modulo a null $\mu^h$ set, which then translates immediately as the functional differentiability of a monotonic non-decreasing function $G o F(x)=G(F(x))$ relative to $\mu_F([0,x])=F(x)$ in the form $d G=G^\prime d F, \ \forall x\in C$, except for a null set relative to $\mu^h$, even as $\frac{dF}{dx}$ is undefined for $x\in C$.

(d)  The Cantor function $F(x)$ must satisfy a sharp bound of the form $x^{s_c}/k\leq F(x)\leq x^{s_c}$ where $k>1$ and the right hand constant 1 are  sharp. The constant $k$, in particular,  relates to the number of closed subinterval produced by each application of the relevant IFS. For example, for a self similar set $C$ with one contraction factor $1/r, \ r\in N$, set of natural numbers, one obtains $k=(r-1)^{s_c}$ \cite{cantor}. Further, the above bound is best possible since $\lim_{x\rightarrow 0} F(x)/x^{s_c} $ does not exist.

It now follows immediately  that the asymptotically extended neighbourhood of $ 0\in \bf R^*$, being homeomorphic to an ultrametric Cantor  set, inherits naturally the above results. More importantly, the set $\bf O$, in fact, enjoys richer analytic and metric properties compared to a general Cantor set.  Indeed, the natural Cantor function that is realized  as the scale invariant effective (ultrametric) value  $X(\xi)$ of an asymptotic variable $x$ must respect a tight bound as in (d). However, {\em recalling the higher level smoothing property in presence of duality structure as stated in Proposition 2.2}, we now have, instead, the equality $X(\xi)/\xi^{s_c}=\phi_0(\xi)(1+\phi_1(\xi)), \ 0<\xi<1$ in the limit $\xi\rightarrow 0$ where $\phi_0(\xi)(\neq 0)$ is almost constant in the vicinity of $\xi=0$ and $\phi_1(\xi)$ is a  second generation Cantor function that exists also as non-vanishing constant  with possible variations at most on a Cantor set $C_0$ of Hausdorff dimension $s_0<s_c$. We remark to emphasize that such an equality is not available in ordinary analysis. We thus have

\begin{theorem}
The duality induced scale invariant renormalized effective variable $X(\xi), \ \xi<1$ awards a regular measure $\mu_X$ defined by $\mu_X([0,\xi])=X(\xi)$ in  the duality enhanced secondary neighbourhood of 0 that is homeomorphic to a Cantor set $C$ with Hausdorff dimension $s_c$. The measure $\mu_X$ is  absolutely continuous with respect to the corresponding Hausdorff measure $\mu_H=H^{s_c}$ so that $d\mu_X=d\mu_H$, in the sense of Radon-Nikodym, modulo a null set relative to $\mu_H$. Finally, $X(\xi)=\xi^{\beta(\xi)}=\xi^{s_c}\phi(\xi), \xi\rightarrow 0$ where the almost constant  function $\phi(\xi)$ can vary only on a gap of the  Cantor set. This equality is valid almost every where on a second generation Cantor set with lower Hausdorff dimension.
\end{theorem}

As a consequence of this theorem, above measure theoretic differentiability  gets inherited by the associated function space $\chi({\bf R^*}, d X)$, the space of smooth functions in $\bf R^*$ relative to the Lebesgue-Stieltjes measure $d X(\xi)$. The transition from the original linear variable $x$ in the primary connected interval $[-\delta,\delta]$ to the scale invariant nonlinear variable $X(\xi)$ living in a totally disconnected, perfect, subset $C$ of the secondary interval $[0,\tilde\delta]$ on the right hand side of 0 is achieved in a smooth manner. In the next Sec. 2.4, we discuss the salient features of this novel differentiable structure.

\subsection{Function spaces: Differentiability }

It is clear from the outset that the spaces of continuous and continuously differentiable functions on $\bf R$ are also $\bf R^*$ continuous and continuously differentiable, since $\bf R^*$ naturally coincides with $\bf R$ at a finite scale $\sim \ O(1)$. However, exploiting the duality enhanced asymptotic uniformization property, the spaces of continuous functions in ${\cal C}^*(\bf R^*)$ and that of continuously differentiable functions ${\cal C}^{*1}(\bf R^*)$ on the extended space $\bf R^*$ are much  richer and includes a large family of standard noncontinuous and nondifferentiable functions respectively. We begin by first examining the structure of duality enhanced extensions of a function in $\bf R$. To this end let us recall the salient features of the structure of the extension of a real variable $x$.

\begin{definition}
1. Let $\delta$ be a privileged first generation scale. Relative to this scale every (finite) real variable $x\in \bf R$ has a unique one parameter asymptotic extension $x^*=x\pm \delta e^{-X(\xi)}, \ x^*\in \bf R^*$, up to a specific Cantor set $C$ of Hausdorff dimension $s$, that is assumed to span the extended neighbourhood of 0 in $\bf R^*$ viz, ${\bf O}=-C\cup C$. The parameter $X=\xi^{-\beta(\xi^{-1})}$ is the renormalized effective value of an asymptotic $x\rightarrow 0$ in $\bf R$ and has the property of the associated Cantor function, that is constant almost every where in the secondary interval $[0,\tilde\delta], \ \tilde \delta\approx \delta\log\delta^{-1}$, with O(1) variation at the points of $C\subset [0,\tilde\delta]$. The transformation $x\mapsto X(\xi)$, in the neighbourhood of $x$ in $\bf R^*$,  is defined by the renormalization group transformation $X(\xi)=\underset{\delta\rightarrow 0}\lim \bigg |\log |x^*-x|/\delta\bigg |$.

2. For $|x|\rightarrow \infty $ the corresponding renormalized extension, however, is given by $|x^*|={\delta}^{-1} e^{X^{-1}}$. The renormalized variable is called a jump variable as it varies in $\bf O$ by countable number of discrete jumps.
\end{definition}

The extension of $x$ considered as above is dependent to the choice of the asymptotic scale $\delta$ and the totally disconnected Cantor set $C$ in the secondary neighbourhood of the form $\delta[1,1+\tilde\delta]$ of the scale $\delta$, that reduces, for the effective logarithmic variable $\log x/\delta$, approximately to the interval $[0,\tilde\delta]$.  It may, however, be recalled that the duality enhanced  asymptotic neighbourhood $\bf O$ of $\bf R^*$ can , in fact, be much more complicated than a simple Cantor set having a uniform Hausdorff dimension.

\begin{definition}
A real valued function $f(x)$ on a set $A\subset \bf R$ has a natural asymptotic extension defined point wise by $f^*(x^*)=f(x)\pm \delta e^{-F(X)}$, that extends $f$ onto the asymptotic neighbourhood $\bf O$ for an appropriate $F(X), \ X\in \bf O$, in a way consistent with the duality structure. In other words, the point wise visibility norm of the extension $f^*(x^*)$ is defined by $v(f^*(x^*)):=v(f^*(x^*)-f(x))=\underset{\delta\rightarrow 0}\lim \bigg |\frac{\log |f^*(x^*)-f(x)|/\delta}{\log \delta^{-1}}\bigg |$ (since $v(f^*(x^*)-f(x))={\rm max}\{v(f^*(x^*)), v(f(x))\}$ and $v(f(x))=0$) that relates to $F$ by $v(f^*(x^*)) \log\delta^{-1}= F_0(X(\xi))(1+o(1)), \ F_0(X)=F(X)-F(0)$ so that the  associated renormalization group transformation $f(x)\mapsto F_0(X(\xi))$ has the form $F_0(X(\xi))=\underset{\delta\rightarrow 0}\lim \bigg |\log |f^*(x^*)-f(x)|/\delta \bigg |$.
\end{definition}

\begin{remark}{\rm The definitions 2.10 and 2.11 make use of Proposition 2.2 in a salient way.
}
\end{remark}

\begin{definition}
A function $f(x)$ that is continuous at $x_0\in \bf R$ has a  continuous extension $f^*(x^*)$  when the renormalization group induced effective function $F(X)$ is continuous in the asymptotic neighbourhood $\bf O$ of $x^*$. A most natural candidate for the continuous extension is through a self similar replication of $f$, defined by $f^*(x^*)=f(x^*)$, on the asymptotic prolongation $\bf O$; i.e. $F(X):=f(X)$ for $X\in \bf O$.

Alternatively, the extension $f^*(x^*)$ is continuous at $x^*$ if the point wise visibility norm of $f^*(x^*)$ is continuous there.
\end{definition}

 A function with discontinuities may, however, be realized, thanks to the asymptotic uniformization (Proposition 2.2), as continuous.

\begin{definition}
A function $f(x)$,  discontinuous at $x_0\in \bf R$, is asymptotically continuous if given $\epsilon>0, \ \exists \ \eta>0$ such that  $|x_1-x_2|<\eta, \ \Rightarrow v(f^*(x_1^*)-f^*(x_2^*))<\epsilon$, even as $|f(x_1)-f(x_2)|>\epsilon$, where $x_1^*$ and $x_2^*$ are corresponding extensions of $x_1, \ x_2$, in the $\delta-$ neighbourhood of $x_0$, respectively.
\end{definition}

It follows immediately that a continuous function $f(x)$ at a point $x$ is asymptotically continuous. The converse of course is not necessarily true.

\begin{proposition}
The $f(x)$ is asymptotically continuous at $x_0$ if and only if $|F(X(\xi_1))-F(X(\xi_2))|<\epsilon$, where $\xi_i$, respectively, are  the corresponding limiting rescaled  variables $\delta/\tilde x_i<1, \ i=1,2$, $\delta\rightarrow 0, \ \tilde x_i=x_i^*-x_0\rightarrow 0$, i.e. $\xi_i$ are two limiting asymptotic variables in $\bf O$.
\end{proposition}

\begin{proof}
Let $v(f^*(x_1^*))>v(f^*(x^*_2))$.

Then ultrametric inequality tells that $v(f^*(x_1^*)-f^*(x_2^*))=v(f^*(x_1^*))<\frac{\epsilon}{2\log \delta^{-1}}$ (say). It follows from Definition 10 that $|F(X(\xi_1)|<\frac{\epsilon}{2}$. As a consequence, $|F(X(\xi_1))-F(X(\xi_2))|<2|F(X(\xi_1))|<\epsilon$.

Conversely, $|F(X(\xi_1))-F(X(\xi_2))|<\epsilon_1, \ \Rightarrow \ \bigg\vert\ |F_0(X(\xi_1))|-|F_0(X(\xi_2))| \ \bigg\vert<|F(X(\xi_1))-F(X(\xi_2))| $ $<\epsilon_1, \ \Rightarrow \  |F_0(X(\xi_1))|<\epsilon_1+|F_0(X(\xi_2))|, 
\ \Rightarrow \ |F_0(X(\xi_1))|(\log\delta^{-1})^{-1}<(1-r)\epsilon+|F_0(X(\xi_2))|(\log\delta^{-1})^{-1}, 
$ $\Rightarrow \ v(f^*(x^*_1)) $ $<$ $(1-r)\epsilon+v(f^*(x^*_2))$, where $v(f^*(x^*_i))=\frac{|F_0(X(\xi_i))|}{(\log\delta^{-1})}$, $ i=1,2$, $(1-r)\epsilon=\epsilon_1 (\log\delta^{-1})^{-1}$ for a suitable $r: 0<r<1$ and $v(f^*(x^*_2))<v(f^*(x^*_1))$. by assumption. Finally, one chooses $r$ as above, so that $ \Rightarrow v(f^*(x_2))=r v(f^*(x^*_1))$. As a consequence, $v(f^*(x^*_2)-f^*(x^*_1))=v(f^*(x^*_1))<\epsilon$.
\end{proof}

\begin{example}{\rm
Consider the step function $f_s(x)=0, \ x<1$ and $f_s(x)=1, \ x>1$ and a $\delta-$neighbourhood $(1-\delta,1+\delta)$ of $x=1$. In the limit $\delta\rightarrow 0$, the point $x=1$ experiences an asymptotic prolongation to a Cantor set $1^*=1+{\bf O}, \ {\bf O}\subset [-\tilde\delta,\tilde\delta]$ with Hausdorff dimension $0<s<1$ equipped with the associated renormalized effective variable $X(\xi)=\xi^s\phi(\xi)$ (by Theorem 2.3), that has the property of the almost constant Cantor function that varies only at the points of the concerned  Cantor set. The function $f_s$ can have the natural asymptotic continuation $F_s(X(\xi))=X(\xi)$ so that the step function is realized as an asymptotically continuous function by extrapolating the function linearly in the asymptotically prolonged secondary interval $[0,\tilde\delta]$ by the effective variable $X$, $\tilde \delta\approx \delta\log\delta^{-1}$, such that the first generation interval $[-\delta,\delta]$ is identified with 0.

A few remarks are in order here.

1.  The Hausdorff dimension $s$ being arbitrary, the asymptotic continuity of the step function is actually facilitated by a one parameter ( $s$) family of effective variables $X_s$ (of course, up to a small constant multiple). Noting that $\inf s=0$ and $\xi^s\approx (1+s\log \xi^{-1})^{-1}$ for a small but fixed $s<<1$ and $\xi\rightarrow 0$, a unique logarithmic extrapolation by an inverse logarithm of the form $X=k(\log \xi^{-1})^{-1}$ is available for achieving asymptotic continuity of the step function.

2. Real functions with finite number of finite jump discontinuities in a bounded interval are asymptotically continuous in the asymptotic neighbourhood of every point of discontinuity.

3. Functions with infinite jump or essential discontinuities such as $x^{-1}$ or $\sin x^{-1}$ are not, in general, asymptotically continuous as long as one works in the class of simple interpolation ( extrapolation) formulas such as polynomial interpolation. More general renormalization group transformations  might be needed for effecting higher order uniform behaviour in such pathological cases, as, for example, in eliminating divergences in nonlinear differential equations \cite{dpp,golden}. In the present paper, we, however, limit ourselves to simpler problems involving self similar replication and/or polynomial extrapolations.
}           $\Box$
\end{example}

The definitions of asymptotically uniform continuity and equicontinuity can now be introduced following standard methods. We, however, consider the concepts of asymptotic differentiability and  replication/proliferation of differential equations in asymptotically prolonged set $\bf O$.

\begin{definition}
A function $f(x)$ differentiable at $x$ (so that $Df(x):=\frac{df}{dx}$) is said to be asymptotically (jump) differentiable if $\underset{x^*\rightarrow x}\lim \ \frac{v(f^*(x^*)-f(x))}{v(x^*-x)}$ exists and is denoted by $D_J f(x)$.
\end{definition}

\begin{lemma}
A differentiable $f(x)$ is asymptotically differentiable if and only if 
$$D_J f(x)= \underset{\xi\rightarrow 0}\lim \ 
\frac{F(X(\xi))-F(0)}{X(\xi)}=\frac{dF}{dX}$$
 exists.
\end{lemma}

The proof follows directly from Definitions 2.10 and 2.11. It follows also that the differentiability relative to the linear variable $x$ extends asymptotically on the totally disconnected prolongation set $\bf O$ relative to the renormalized effective variable $X$.

\begin{proposition}
Let $f(x)$ be differentiable at $x$, then $F(X(\xi))=X(\xi)\mp\log f^{\prime}(x+\tilde\xi), \ \tilde\xi\in [0,\delta]$ and $\delta\rightarrow 0$. As a consequence, a differentiable function has a universal small scale behaviour characterized by the linearization of renormalized effective value $F_0(X)=X$, up to an inessential constant (depending on $x$) in the vicinity of $0 \in \bf O$. Alternatively, by a slight redefinition, viz, $F_0(X(\xi))=\underset{\delta\rightarrow 0}\lim \bigg |\log |f^*(x^*)-f(x)|/f^\prime(x+\xi)\delta\bigg |$, where $\xi\rightarrow 0$ as $\delta\rightarrow 0$, in the secondary interval, we have $F_0(X)=O(X)$.
\end{proposition}

\begin{remark}{ The approximate linearity of a smooth function $f(x)$ for small $x\approx 0$ is inherited by the renormalized function $F(X)$ in the renormalized effective variable $X$.
}
\end{remark}

\begin{proof}
By definition, $f^*(x^*)=f(x^*)=f(x+\delta e^{-X})$ (corresponding to self similar continuation for smooth functions) and $f^*(x^*)=f(x)\pm \delta e^{-F(X)}$. Asymptotic increment $\delta e^{-X}$ is continuous on the secondary asymptotic interval $[0,\tilde\delta]$. Hence by the mean value theorem $f(x+\delta e^{-X})=f(x)+\delta e^{-X}f^\prime(x+\tilde\xi)$ where $x+\tilde\xi\in (x,x^*)$ so that $ \xi\in [0,\tilde\delta]$. Further, $\xi\rightarrow 0$ with $\delta\rightarrow 0$ in the secondary interval. Hence the lemma.
\end{proof}

\begin{remark}{\rm A differentiable function $f(x)$ in $\bf R$ has the natural self similar prolongation $f(X)$ for an effective variable  $X$ on the corresponding prolongation $\bf O$, by translation invariance. For such a self similar prolongation, {\em a differentiable function is necessarily asymptotically differentiable } i.e. $D_J f(x)=\frac{d f(X)}{d X}$ exists, by inheritance, from that of  $\frac{d f(x)}{d x}$ in $\bf O$.
}
\end{remark}

\begin{definition}
A continuous function $f(x)$, nondifferentiable at $x$, is said to be asymptotically differentiable for right renormalized quantities $X$ and $F(X)$ so that $\frac{dF}{dX}$ exists on $\bf O$.
\end{definition}

\begin{example}{\rm {\em Differentiable functions:} By Definition 2.10, the basic independent variable $x$ has the renormalized asymptotic extension $X=\underset{\delta\rightarrow 0}\lim |\log |x^*-x|/\delta|=\xi^{\beta(\xi)}, \ 0<\xi<<1$, that replaces the original linear variable $x$, in either of the asymptotic limit $x\rightarrow 0$ or $\infty$,  and acquires the status of the independent variable. As a consequence, the trivial differentiability of $x$ with respect to itself is inherited by $X$ relative to itself, although it is only differentiable almost every where relative to $x$. In other word, the trivial scale invariant Cauchy problem $x\frac{dy}{dx}=y, \ y(1)=1$ on the any bounded interval in $\bf R$  gets replicated in the asymptotic sector $x\rightarrow \infty$, for instance, in the form $\frac{dY}{dX}=1, \ Y(1)=1, \ X\in [0,1]$ (for more details, see Sec. 3.1).

In the same vain, a differentiable function, for instance, $y=e^x$ for $|x|<\infty$ has the point wise extension defined by the renormalized variable $Y=X, \ X\in [0,1]$. More generally, a differentiable function $f(x)$ is asymptotically differentiable.
}               $\Box$
\end{example}

\begin{example}{\rm {\em Non-differentiable functions:} (a) Consider first  the step function $f_s$ of Example 2.8. Natural interpolation function in the asymptotic prolongation $1\mapsto 1+\bf O$ being $X(\xi)$, the Cantor function that also is the precise renormalized effective value of the primary independent variable $x$, the classical non-differentiability at $x=1$ is extended as asymptotically differentiable. As a consequence, the step function $f_s$ is realized not only as a continuous function through the linear interpolation by $X$, but $f_s$ is also differentiable in the asymptotic prolongation $\bf O$ relative to the asymptotic renormalized independent variable $X$. In fact $\frac{df_s(X)}{dX}=1, \ X\in 1+\bf O$, when $\frac{df_s(x)}{dx}=0$, every where except $x=1$.

( b) Consider the class of functions represented by power series of the form $f(x)=\sum a_n x^{ns^{\prime}}$ in the variable $x^{s^{\prime}}$. Recalling that $x\ \rightarrow \ X=\xi^s$ for an asymptotic $x\rightarrow 0$, $f(x)$ is asymptotically differentiable in $\bf R^*$ for $s^{\prime}\ge s$.
}   $\Box$
\end{example}

For definiteness, let us denote by $\tilde{\cal C}_0(\bf R)$ the augmented space of functions in $\bf R$ consisting not only of continuous but also of  bounded step functions on a measure zero subsets of $\bf R$. It turns out that the space $\tilde{\cal C}_0(\bf R)$ gets extended naturally on $\bf R^*$ as the space of asymptotically differentiable functions.

\section{Applications}
\subsection{Differential Equations: New asymptotics}

The differential operator $D=\frac{d}{d x}$ acting on smooth functions on $\bf R$ is extended over the associated smooth function space $\chi(\bf R^*)$ of $\bf R^*$ as the direct sum $D^*=D\oplus D_J$ where $D$ acts on the smooth function space $\chi(\bf R)$ and the asymptotic derivation operator $D_J$ acts nontrivially only on the  space $\chi_J(\bf R)$ of asymptotically differentiable functions on $\bf R$ respectively. In other words, the action of the original derivation operator $D$ on a function of $\chi(\bf R)$ is replicated self similarly on the asymptotic prolongation set $\bf O$, thus leading to new nontrivial asymptotics, even in the case of ordinary smooth functions. Analytically, $D^*f^*(x^*)=Df(x)\oplus \delta D_J e^{-F(X)}, \ x^*\in {\bf R^*}, \ x\in {\bf R}, \ X\in \bf O$. As a consequence, we have

\begin{proposition}
Let ${\cal L}(D)$ be a linear homogeneous differential operator acting on smooth function space $\chi(\bf R)$ of $\bf R$. Then
$${\cal L}^*(D^*)f^*(x^*)={\cal L}(D)f(x)\oplus \delta {\cal L}(D_J) e^{-F(X)}.$$ Consequently, the original equation  ${\cal L}(D)f(x)=0$ on $\bf R$ is replicated self similarly on asymptotic sector of the form $\bf O$ as ${\cal L}(D_J) {e^{-F(X)}}=0$. The solution space $\chi(\bf R)$ accordingly is extended to $\chi({\bf R^*})=\chi({\bf R})\cup \chi_J({\bf R})$, when one considers the natural self similar renormalized extension $F(X):=f(X), \ X\in \bf O$.
\end{proposition}

Proof follows easily from the linearity of the differential operator ${\cal L}(D)$ and the Definitions 2.10 and 2.11 on the self similar renormalized extensions. Notice also that the operators $D$ and $D_J$ acts trivially on functions of $\chi_j(\bf R)$ and $\chi(\bf O)$ respectively, since $ f\in \chi_J(\bf O)$ is independent of $x$ and $f$ on $\bf R$, on the other hand,  is constant relative to $X$.  $\Box$

\begin{example}{\rm
1. We have seen that the basic independent variable $x$ is extended as $x\mapsto X=\xi^{\beta(\xi)}, \ \beta(\xi)<1$, so that either in the asymptotic limits $x\rightarrow 0$ or $x\rightarrow \infty$, the linear asymptotic $x$ is replaced  as the scale invariant power laws $\xi^{\beta(\xi)}$ or $\xi^{1/\beta(\xi)}$  respectively (the power law behaviour at infinity follows from that  in the vicinity of $x=0$ by duality). Notice further that the corresponding differential equation $\frac{dy}{dx}=1$ (as a Cauchy problem with data $y(1)=1$, say) gets proliferated on $\bf O$ by self similar replication viz, $\frac{dY}{dX}=1, \ X\in \bf O$, in conformity of Proposition 2.4.

Incidentally, let us note that the above smooth proliferation of the simplest linear ODE actually hides the irregular variation of a Cantor function of the form $X$ at the points of the Cantor set concerned, viz, $C$. The present scale invariant representation of the same in  $\bf O$ does reveal the underlying intermittent  variation in the language of the calculus on a non-archimedean space, viz, the variable $X$ remains locally constant on a gap $I_i$ of the associated Cantor set $C$ when the exponent $\beta(\xi)$ varies as $\beta(\xi)=\frac{\beta_i}{\log \log \xi^{-1}}(1+o(1))$, which coincides with the variation exposed in Lemma 2.9 for a vary small exponent $b_0$. In the neighbourhood of a point in the Cantor set $C$, the exponent is frozen to a fixed value,  $s$, the Hausdorff dimension of $C$, so that one now obtains $X=\xi^s$. In the next section, we verify that this variation of $X$ on $C$ is indeed consistent with the class of Cantor functions associated with the homogeneous Cantor sets.

2. Let us recall the well known fact that any smooth function on $\bf R$ is interpreted as a solution of a {\em linear} ODE (of appropriate order). In virtue of the Proposition 3.1,  each such equation is replicated on the asymptotic prolongation set $\bf O$. As a consequence, a simple periodic function, for instance, $\sin x$, would  tend asymptotically to a complicated  periodic pattern in the form $\sin X^{-1}$ for sufficiently large $x\rightarrow \infty$, by duality. This is justified by the following Corollary,  in which we sketch in detail the actual mechanism of the asymptotic proliferation considered here.}  $\Box$
\end{example}

\begin{corollary}
Consider the time dependent harmonic oscillator (HO) equation $$\ddot x+a(t) x=0, \ t\in\bf R.$$ The corresponding self similar replica on $\bf O$ has the form $$\frac{d^2 X}{d T^2_{\delta}} + a(T_{\delta}) X=0 $$ where $T_{\delta}\approx T\delta^{-1}\log \delta^{-1}$ and $a^*(T_{\delta})=a(t^*), \ t\mapsto t^*=\delta^{-1}\times\delta^{\pm T}, \ x\mapsto x^*=x+\delta\times\delta^{\pm X}$ relative to the asymptotic scale $\delta$. In other words, the ordinary time dependent harmonic oscillation would switch over to an equivalent irregular oscillation described, instead, by the self similar replica equation on the totally disconnected set $\bf O$, in the renormalized effective variables $T\sim O(1)$ and $X\sim O(1)$  for a sufficiently large time $t\mapsto t^*$ living in $\bf R^*$ and $0<\delta<<1$.
\end{corollary}

\begin{proof}
We present two proofs. First one is the direct Corollary to Proposition 2.4 and follows by  continuity (the prolongation of point set to a set of the form $\bf O$ and the associated effective renormalized variables are a continuous in the ultrametric topology of $v$.)

For the second proof  we proceed in steps; this  highlights the actual steps that one needs to follow in the proof of the above Proposition in the general setting. (a) The original HO equation in $\bf R$ is extended in $\bf R^*$ by continuity as
$$\ddot{ x^*}+a^*(t^*) x^*=0, \ t^*\in {\bf R^*}$$ where $\dot x^*$ now means derivation in $t^*$ and $a^*(t^*)=a(t)+\delta e^{-a(T)}$. (b) This extension decouples into two separate independent  components, one being the original oscillator component in the real variables  $t$ and $x$ and other being the self similar replica continued in $\bf O$ in the renormalized quantities $X$ and $T_{\delta}$. Homogeneity of the oscillator equation concerned allows one to cancel a common factor $\delta\log \delta^{-1}$. (c) To achieve this goal, one needs to work with    linearized variables of the form  $x^*\approx x+\delta+ X\delta\log \delta^{-1}$ and $t^*\approx t+T_{\delta}, \ t\sim \delta^{-1}$  and neglect terms in higher degree in $\delta$ in comparison to $\delta\log\delta^{-1}$. (d)  Replica HO equation in $\bf O$ is realized for sufficiently small, but nevertheless $ O(1)$, renormalized variables $X$ and $T$ so that the linearization is valid. It follows, as a result, that, $T_{\delta}\sim O(\delta^{-1}\log{\delta^{-1}})$.
\end{proof}

\begin{example}{\rm Consider the constant frequency HO equation $\ddot{ x}+ x=0$ where the time variable $t$ is appropriately rescaled to normalize the frequency to the value 1. In this setting the natural scale of the problem is 1. The elapse of time is thus counted in natural numbers $t\sim N>>1$ and asymptotic variables are introduced as $t^*=N\times N^{ T}, \ x^*=x+N^{-1}\times N^{ X}$ so that $X=Sin \ (N\log N) \ T$ when we choose $x=Sin \ t$ as the fundamental oscillation. The self similar proliferation considered here should, however, be asymptotically sub-dominant and remains imperceptible relative to the ordinary oscillation $x$ for all time. In fact, for sufficiently large $t$, small scale variation in $T$ would remain unobservable compared to the dominant oscillation $Sin \ t$, unless special care is taken to eliminate the dominant one.  {\em The key feature of the present analysis is, however, to highlight the existence of  parallel streaks of irregular oscillations in the form $X=Sin \ (N\log N)\ T$ for $T\sim O(1)$ for each sufficiently large $N$, in the visible asymptotic sector of the form $t\in N(1-\log N, 1+\log N), \ N\rightarrow \infty$}, that could only be manifest in the renormalized variables $T$ and $X$. In other words, one verifies easily that the original harmonic oscillation is retrieved in the form $$\frac{d^2 X}{d T^2_{N}} +  X+\frac{1}{\log N}=0 $$ in renormalized variable $X=\frac{\log |x^*-x|N}{\log N}$ and the normalized renormalized time $T_N=N\log N T, \ T=\frac{\log t^*/N}{\log N}$, for sufficiently large $N$, so that $1/\log N=o(1)$.}  $\Box$
\end{example}

It is important to realize that as long as one continues to work in the frame work of linear time $t$ and identifies $T\propto t$, the oscillation, in the above example, is purely harmonic even for late times. However, in the present extended formalism, $T_N$ is a Cantor's devil staircase, so that the original moderate time smooth sinusoidal oscillation would experience a late time smooth transition into  an irregular devil's sinusoidal oscillation, of course, relative to the linear time $t$, but, nevertheless, interpreted as smooth in the renormalized effective variable $T_N$. Consequently, {\em the ordinary harmonic oscillation is endowed with rich potential of evolving into novel asymptotic patterns that might even  become visible in right linearized effective variables close to the asymptotic sector in the context of duality structure}.   The residual, irregular sinusoidal oscillation should become noticeable if one designs a controlled experiment on harmonic oscillator with the intention of  zooming into  late time oscillation at time $t\sim O(N)$ by freezing the original motion at $t=N$. Conventionally, one should have a typical point $(x(N),\dot x(N))$ in the oscillator phase plane at time $t=N$. In the context of the present formalism, one, instead, expects to see an extended phase orbit $(x(T_N), \dot x(T_N))$ corresponding to the temporal prolongation set $T_N\in (0,N \log N)$ that becomes available as a nonclassical neighbourhood of the frozen instant $t=N$.

The present novel interpretation should have significant applications, for instance, in  turbulence, bio-rhythms, financial variations and other analogous complex systems. In the next section we first show how a linear plane traveling wave can acquire late time nonlinear characteristics. Next, we derive nontrivial power law attenuation of linear wave propagation in a viscoelastic lossy medium directly from the asymptotic scaling, bypassing thereby the need for invoking fractional differential lossy operators as in vogue in the current literature \cite{lossy1, lossy2, lossy3}.

We close this section with the following observation. As the evolution of certain system continues up to an asymptotic time, determined by the characteristic scale $\delta$ of the system, the subsequent evolution is governed by the irregular flow of the jump mode renormalized variable in the form $t\mapsto t^*\approx \delta^{-1}(1+T\log\delta^{-1})$. This transition could be achieved by eliminating the frozen dominant motion in the asymptotic limit of the linear time $t$ to $\delta^{-1}$, the subsequent evolution is being controlled by the renormalized time $T$ that flows intermittently as $T\sim n/\log\delta^{-1}$ in the unit of $(\log\delta^{-1})^{-1}$ for a fixed $\delta$.

\subsection{Novel plane wave asymptotics}

We begin by considering plane wave solution of the form $Ae^{i(-\omega t+kx))}$ of the one dimensional linear wave equation

\begin{equation}\label{wave1}
\frac{\partial^2 u}{\partial t^2}={c^2}\frac{\partial^2 u}{\partial x^2}
\end{equation}
corresponding to non-dispersive traveling wave form satisfying $k^2=\omega^2/c^2, \ \omega>0$ and $-\infty<k<\infty$. In presence of duality structure, this simple traveling wave should experience late time instabilities analogous to wave propagation in a fractal medium. For sufficiently large time $t\sim O(1/\epsilon), \ 0<\epsilon<<1$ the linear flow is inherited, instead, by a renormalized effective time variable $0<T<1$ where $t=\epsilon^{-1}\times \epsilon^{-v(t)} \ \Rightarrow \epsilon t\approx 1 +T$, so that $T=v(t)\log \epsilon^{-1}=\eta^{\alpha}$ (by Theorem 3.1) plays the role of new effective measure in the duality enhanced prolongation set of the form ${\bf O}_t$ on the $t$ axis. Here, $0<\eta<1$ is an limiting rescaled  (dimensionless) variable in ${\bf O}_t$, $\alpha(\epsilon)$ is  slowly varying with $\eta$ and corresponds to the Hausdorff dimension of fractal Cantor set representing ${\bf O}_t$. As a consequence, the linear differential measure $dt$ on the time axis is replaced asymptotically as $dT$.

In association with the above asymptotic time scaling, one also envisages an analogous asymptotic space  scaling (renormalization group) transformation: instead of following the traveling  wave at greater distances as $t\rightarrow \infty$, one instead could zoom into the asymptotic neighbourhood of a spatial point $x_0$, say, by duality, to investigate the small scale wave properties more carefully so that $\Delta x=x-x_0\rightarrow 0$. The asymptotic space rescaling is now defined by $x=x_0+\epsilon\times \epsilon^{-v(x)}\approx x_0+\epsilon +X, \ X=v(x)\epsilon\log \epsilon^{-1}=\xi^{\beta(\epsilon)}$ so that $dx\mapsto dX$ on the spatial prolongation ${\bf O}_x$ for the rescaled  (dimensionless) spatial variable in (0,1) and $\beta(\epsilon)$ is the slowly varying  Hausdorff dimension of the Cantor set representing ${\bf O}_x$ at the scale $\epsilon$. Such a concomitant spatio-temporal rescalings is called, for definiteness, the {\em coherent correlated (spatio-temporal) duality} that exists naturally in the asymptotic sectors of an evolving system in the context of the present duality enhanced analytic framework. In 3 dimension, the spatial prolongation of the neighbourhood of 0 is, in fact, a 3 dimensional space spanned by the respective effective variables $X, \ Y, \ Z$ with associated scaling laws dictated by the respective scaling exponents $\beta_x(\epsilon), \ \beta_y(\epsilon), \ \beta_z(\epsilon)$.

In the one dimensional wave pattern that we are concerned here, the wave equation (\ref{wave1}) now replicates self similarly in the asymptotic prolongation ${\bf O}_t \times {\bf O}_x$ in the form
\begin{equation}\label{wave2}
\frac{\partial^2 u}{\partial T^2}={C^2}\frac{\partial^2 u}{\partial X^2}
\end{equation}
where $C=c/\epsilon$. The deformed plane wave $\sim e^{i(-WT + KX})$ in  self similar variables $T, \ X$ now has the dispersion law $W=C|K|$, where $W/C=(\omega/c)^{\alpha(\epsilon)}$ and $K=k^{\beta(\epsilon)}$ that follow from dimensional analysis. As a consequence, one obtains finally the duality induced deformed dispersion law, $\omega=c |k|^{\beta/\alpha}$, as $t\sim 1/\epsilon\rightarrow \infty$.

Novel patterns in the wave propagation  are  manifest from the fact that the scaling exponents are non-integers $0<\alpha, \ \beta <1$. The constancy of phase velocity $\frac{d\omega}{d k}=c$ for moderate time evolution is broken spontaneously  for sufficiently large time because of late time emergent nonlinearity in the dispersion law. The exact asymptotic phase velocity is determined by the coherent spatio-temporal structure with scaling exponents $\alpha$ and $\beta$. In the special case when $\alpha=\beta$ linear dispersion law continues to hold even on the asymptotic prolongation structure. Such deceptively linear coherence may be interpreted as a result of a very special spatio-temporal coherence due to judicious cancellation of nonlinear scaling exponents. As a consequence, {\em an observational verification of the  linear, rather than the expected nonlinear, dispersion law even at late time need not necessarily be considered as a failure of the present formalism, rather this  late time linearity might in fact reveal a deeper spatio-temporal coherence due to duality}.

{\bf Power law attenuation: }
\vspace{0.2cm}

Power law attenuation of  the traveling plane (acoustic/elastic) wave in the form $k\omega^y$ for $0\le y\le 2$  is observed abundantly in many complex systems such as biological tissues, rocks, polymers, sea-shore sand dunes,  medical imaging such as ultrasound or magnetic resonance elastography and many others\cite{lossy1,lossy2,lossy3} . Applications of  fractional loss operators, for instance, \cite{lossy2} in analyzing the observational data is predominant in  current literature. We now show that the present duality enhanced formalism is potentially quite rich in explaining the origin of the power law attenuation in a unified manner based instead on the classical third order  viscoelastic (dispersive) wave equation

\begin{equation}\label{waveloss}
\frac{\partial^2 u}{\partial t^2}={c^2}\frac{\partial^2 u}{\partial x^2} + \nu\frac{\partial}{\partial t}{{\frac{\partial^2 u}{\partial^2 x}}}
\end{equation}
where $\nu$ relates to the kinematic viscosity (equivalently, a relaxation time) characterizing the medium. The associated dispersion law is obtained as $\omega^2=c^2k^2+(i\omega\nu)k^2$, leading to the standard quadratic attenuation $y=2$, that is observed to arise in most of the simple systems such as air, water, water-oil interface etc \cite{lossy2}.

A complex system such as  biological tissue, on the other hand, is an emergent phenomenon, that is evolved in Nature, following certain well defined rules (that might not be known explicitly),  asymptotically  from some (equivalent class of) simple initial states. We remark that the complex late time evolution must be independent of any specific initial condition that was considered in defining the simple system. Keeping in view of  this reasonable assumption, description of attenuation in a complex lossy system can   be derived {\em as a signature of the underlying  asymptotic prolongation}, respecting the principle of {\em coherent correlated duality} introduced in the context of wave equation (\ref{wave1}), directly from  the classical viscous wave equation (\ref{waveloss}).

By simple rescaling, (\ref{waveloss}) is first rewritten as one involving dimensionless time and space variables $t$ and $x$ so that the propagation velocity $c=1$ and the dimensionless viscosity $\nu$ introduces a new time scale. For a bio-fluid the viscosity $\nu$ is large $\nu>1$ and so the late time scaling, in linearized form, is written as $t=\nu^n(1 +T), \ T=v(t)\log \nu^n\sim O(1)$. The associated coherent correlated scaling for the space variable now has the form $\Delta x=\nu^{-n}+ X, \ X=v(x)\nu^{-n}\log \nu^n\sim O(\nu^{-n})$ , so that the original (dimensionless) equation (\ref{waveloss}) replicates self similarly in the prolongation set ${\bf O}_t\times {\bf O}_x$ in the form
\begin{equation}\label{waveloss2}
\frac{\partial^2 u}{\partial T^2}=\frac{\partial^2 u}{\partial X^2} + \nu\frac{\partial}{\partial T}{{\frac{\partial^2 u}{\partial^2 X}}}
\end{equation}
in the renormalized effective variables $T=\eta^{\alpha(\nu)}\sim O(1)$ and $X=\xi^{\beta(\nu)}\sim O(\nu^{-n})$ with $0<\alpha, \ \beta <1$. The deformed dispersion relation is therefore obtained as $\omega^{2\alpha}=k^{2\beta}-i(\tilde\nu\omega)^{\alpha}k^{2\beta}$ so that $k^{\beta}=\omega^{\alpha}(\frac{1-i(\tilde\nu\omega)^{\alpha}}{1+(\tilde\nu\omega)^{2\alpha}})^{1/2}$ where $\nu={\tilde\nu}^{\alpha}$. Clearly, the attenuation in the deformed plane wave mode is due to the imaginary part of $k^{\beta}=k_r^{\beta}-ik_i^{\beta}$ (say), where $k_i^{\beta}\approx \frac{{\tilde\nu}^{\alpha}}{2} \omega^{2\alpha}$ when $\tilde\nu\omega<<1$. In the other case, that is for $\tilde\nu\omega>>1$, we have instead
$k_i^{\beta}={\tilde\nu}^{-\alpha/2}\sin \frac{\pi\alpha}{4} \omega^{\alpha/2}$. Noting that $0<\alpha, \ \beta <1$, both these power law  formulas are consistent with the results derived from the fractional lossy operator approach in literature \cite{lossy1,lossy2,lossy3} and expected to have wide applications in complex material attenuation problems.

To summarize, Gaussian (quadratic) attenuation of the classical dispersive lossy equation (\ref{waveloss}) for simple systems is shown to have been transformed  asymptotically into a deformed plane wave with power law attenuation characteristic of wave attenuation in a complex material such as biological tissues. The complex wave propagation pattern is traditionally explained as the transformation of the ordinary time or space derivatives operators in the dispersive (lossy ) part in  the classical wave equation (\ref{waveloss}) by an appropriate fractional order ($\alpha$, say) derivatives, allowing thereby continuation of ordinary solution space of the problem to include more general (continuous but nondifferentiable) function spaces. The present formalism, on the other hand, invokes  more general such function spaces quite naturally based on the duality structure and the associated renormalized group transformations. Wave propagation in viscous complex systems is described, in this formalism, as an asymptotic problem in right renormalized effective variables (having intermittent fractal characteristics). An advantage in this approach is that the governing evolutionary equation for the complex  system is a self similar replica of the original simpler problem in the appropriately tuned asymptotic renormalized variables, leading to the correct scaling laws, as verified in the above power law attenuation. The free scaling exponents $0<\alpha, \ \beta <1$ give not only the flexibility  in arriving at correct fit with experimental data, but more importantly shed important light on the inherent characteristics of the complex system concerned. Finally, the scaling laws characterizing a complex system are fabricated delicately by the principle of {\em spatio-temporal coherent correlated duality} as formulated in the context of a partial differential equation as correlated asymptotic rescalings of time and space variables leading to coherent spatio-temporal multifractal structures.

 Generalization of analogous results to linear non-homogeneous equations, for instance, a forced HO equation, is obvious. For nonlinear equations, however, the  linearization does not generally apply. In any case, the present duality enhanced analytic framework is expected to have significant applications in nonlinear problems. Nonlinear ODEs, in particular, is known to produce cascades of multiple scales from nonlinear feedback into the system. In \cite{dp1,dp2},  some interesting applications of the formalism into new instabilities in the plasma turbulence and in the problem of estimating nonlinear period orbits in a class of Lienard systems are already studied. More detailed understanding of the precise role of duality structure in nonlinear differential equations require  separate investigations and would be taken in the subsequent paper II.

\subsection{Fractal Sets}

Fractal sets such as  self similar Cantor subsets $C$ of the real line $\bf R$ are examples of emergent complex systems. By definition, a Cantor set $C$ is a compact, perfect and totally disconnected subset of the closed interval $[0,1]$, say. As a consequence, the ordinary differential calculus of $\bf R$ is unsuitable  in developing an analytic framework on  $C$ so that questions regarding variations or flow of a dynamic quantity as a function of a variable $x$ varying purely on $C$ can not be meaningfully analyzed. The  present extended analytic formalism equipped with duality structure  provides a right framework for a theory of differential equations on a fractal set. As a specific example, let us consider the middle third Cantor set $C_{1/3}$. For simplicity of notation, we, however, continue to denote this set as well by $C$.

We begin by first restating ( c.f. Introduction ) the fundamental {\em selection principle} that is being advocated in this   paper (and especially in II):  Every nonlinear complex system  is supposed to enjoy a privileged duality enhanced soft model $\bf R^*$ of the real number system, that happens to be determined by the  characteristic scales of the system. The renormalized effective variable $X$ in the corresponding prolongation of a point of the extension then offers itself as the most appropriate uniformizing variable   so that the said complex system is awarded a differentiable structure. This principle should also be applicable in a converse sense i.e. a soft extension $\bf R^*$ is supposed to provide a natural smooth structure corresponding to a unique  class of nonlinear systems. Recall that the hard model $\bf R$ corresponds to the  smooth structure for the class of differentiable functions, together with smooth linear systems. A few simple applications of this enhanced differentiability is already analyzed in the previous section. Here, we consider a particular application of this idea into a truly nonlinear system, that is, the Cantor set $C$.

To award a smooth structure on $C$, let us begin by noting that a point $x$ in $C$ is {\em asymptotic}, being a member of the limit set of the iterated function system (IFS)  $F=\{F_1, \ F_2 \} : I\mapsto I, \ I=[0,1]$ such that $F_1(x)=x/3, \ F_2(x)=(x+2)/3, \ 0\le x\le 1$ and $C=F(C)=F_1(C)\bigcup F_2(C)$. Exploiting self-similarity, it is sufficient to consider an asymptotic segment $\tilde C_n$ of $C$ as the subset within a closed interval of the form $I_n=[0,3^{-n}]$ so that $C_n\subset I_n$  and $3^{-n}$ for $n>N, \ N$ sufficiently large, is the privileged scale.

Let $x \in \tilde C_n$. Then there exists $x_n\in I_n, \ {\rm such \ that \ either} \ x_n\in I_{n}/3 \ {\rm or} \  x_n\in (2+I_{n})/3$ and $\underset{n\rightarrow \infty}\lim \  x_n=x $. Further, $I_{n}= F_1(I_{n-1})$ so that $I_{n}=F_1(I_{n})\bigcup G_{n}\bigcup F_2(I_{n})$, where, $G_{n}$ denotes the open gap in $I_{n}$.
Now, $x$ is either a boundary point or a limit point of a sequence of boundary points. In the later case $x$ is asymptotic. Otherwise, $x$,  a boundary point, is itself a limit point of a one sided sequence of boundary points of open gaps, either solely from the right or from the left of $x$. As a consequence, each element  $x\in C$ is realized as a limit of a convergent sequence $\{x_n\}$ and hence inherits a uniform (renormalized) visibility effective value $X(\xi)$ from the visible asymptotic elements in $C$ in the vicinity of the  right neighbourhood of 0.

The classical triadic Cantor set $C$ viewed as a subset of a soft model ${\bf R^*}_c$ now  makes a {\em natural selection}: the soft model ${\bf R}^*_c$ is chosen in such a manner as to allow the prolongation set ${\bf O}\subset{\bf R}^*_c $ to inherit the geometric structure of the given nonlinear set $C$. As a consequence,  $\bf O$ would  be precisely the self similar replica of $C$ in the neighbourhood 0. The set $\bf O$ is further equipped with  the renormalized effective variable $X(\xi)$ of an asymptotic element $x\in C\subset{\bf R}^*_c$ that happens precisely an identical copy of the Cantor function measure with support ${\bf O}$. By translation invariance and also by above remark, this effective $X$ now lives   in the duality enhanced  (prolongation) neighbourhood of every point $x\in C$. The differentiability of a function  $f: C\mapsto R$ now is defined, following Sec. 2.4 on the prolongation neighbourhood of $x\in C$ relative to the effective variable $X$.

To see explicitly how duality structure induces a smooth structure on $C$ relative to the associated Cantor function measure, we note that
 for sufficiently large $n$, an element $\tilde x$ in $I_n$ is {\em invisible} relative to the scale $\delta=3^{-n}$ satisfying  $3^n\tilde x\lessapprox 1$, so that $\tilde x$ is an element of $F_2(I_{n})$. Consequently,  $I_n$, at this level of approximation, is equivalent to the point 0, but $I_{n-1}\supset I_n$ is nontrivial, i.e. contains visible asymptotic numbers. Let $x\in I_{n-1}\backslash I_n$ be such a visible number such that $3^n x\gtrapprox 1$. The visible $x$ must, therefore,  by construction, be an element of the open gap $G_{n-1}$, that is immediate adjacent to $I_n$, in $I_{n-1}$. As $n\rightarrow \infty$, the Lebesgue measure of $I_n\subset {\bf R}^*_c$, involving invisible asymptotic numbers, would diverge (c.f. Sec. 2.3), by an application of the asymptotic selection principle, {\em uniquely }
as $3^n\tilde x=\lambda_n 3^{n(i^{-1}_n 2^{m_n} )}$ where integers $m_n$ increases slowly compared to $n\rightarrow \infty$, and $0<i_n\le 2^{m_n-1}$,  $\lambda_n\sim O(1)$, so that the rescaled visible element, viz,  $3^n x=3^{-nX(\xi)}:=3^{-n(i_n 2^{-m_n} )}$ gets the effective value $X(\xi)$ that lives on the prolongation interval ${\cal P}_n:=[0,3^{-n}\log 3^n]$ and remains constant almost everywhere in ${\cal P}_n$, as $n \rightarrow \infty$. The constant values of $X(\xi)$ is given by $X(\xi)=i_n 2^{-m_n}$ on the $i$th gap of the $m_n$th iterative definition of $C$ (that is, of ${\cal P}_n\cap C$). These values clearly extends continuously across the boundaries of the gaps, so that $X(\xi)$ experiences singular variations relative to $\xi$, viz, $\frac{d X}{d\xi}$ being undefined, at the points of ${\cal P}_n\cap C$. As a consequence, $X(\xi)$ corresponds uniquely to the continuation of the Cantor measure with support  $C$ on the prolongation $\bf O$ as $n\rightarrow \infty$.

\begin{remark}{\rm  The emergence of the renormalized effective value $X(\xi)$ under the influence of duality structure proceeds through two step application of the selection principle: of all possible choices of the prolongation set $\bf O$, those remain hidden as latent possibilities in the neighbourhood of 0 in $\bf R$, the primary definition (choice/selection) of the particular Cantor set, here the triadic set $C$,  as the object of investigation, stimulates one to look for that special prolongation ${\bf O}_c$, that appear {\em naturally as the self similar continuation of $C$ in the non-classical, duality enhanced sector}, together with precisely an inherited  smooth structure from the associated Cantor function measure $X(\xi)$. The definition of the renormalized effective value, as described in the above paragraph as the inverse of the  divergent exponential factor in the definition of associated invisible elements, then makes the final selection.
}
\end{remark}

{\em Function spaces on $C$:}

\vspace{0.2cm}

Let $f: C\mapsto \bf R$, when the set $C$ is viewed as a subset of $\bf R$. Such a function, by definition, is identical with the unique continuation $\tilde f(\tilde x)$, viz, $f(x)=\tilde f(\tilde x), \ x\in C, \ \tilde x\in [0,1]$, and, hence,  is almost everywhere constant on $[0,1]$ except at points on $C$ where it would experience changes  that is, nevertheless,  singular. As a subset of $\bf R^*$, the fractal set $C^*$ enjoys a smooth structure relative to the renormalized effective variable $X(\xi)$ in the duality enhanced self similar prolongation neighbourhood ${\cal N}(x)=x+{\bf O}, \ {\bf O}\subset [-\tilde\delta,\tilde\delta]$ of a point $x\in C\subset [0,1]$ where $\tilde\delta=\delta\log\delta^{-1}, \ \delta=2^{-n}, n\rightarrow \infty$. The original function $f$ is assumed to have natural self similar continuation on the prolongation ${\cal N}(x)$.

\begin{definition}
A function $f: C\mapsto \bf R^*$ is jump differentiable at $x\in C\subset \bf R$   if it is asymptotically differentiable in the prolongation ${\cal N}(x)$ and is denoted as $D_J f(x)$. It follows that $D_J f(x)=\frac{d f(X)}{d X}$.
\end{definition}

\begin{example}{\rm
Let $f_c(x)$ be the Cantor staircase function corresponding to the Cantor set $C$. The prolongation of $f_c$ on the prolongation set ${\cal N}(x)$ is clearly $f_c(x)\mapsto f_c(X)=X$, where $X$ is itself the self similar replica of the Cantor staircase function $f_c$ on ${\cal N}(x)$. As a consequence, $D_Jf_c=1$ for each $x\in C$, and hence $D_J f_c(x)=\chi_C(x), \ x\in \bf R$, where $\chi_C(x)$ is the characteristic function of $C$: $\chi_C(x)=1,$ for $x\in C$ and $\chi_C(x)=0$, otherwise. For $x\in {\bf R}\setminus C$, one recovers the standard result, viz,  $D_J f(x)=D f(x)=0$.This result is indeed consistent with Ref.\cite{parvate}.}   $\Box$
\end{example}

\begin{example}{\rm
The Cantor set $C$ is the limit set of the IFS $F$ on [0,1]. This example now shows that {\em  under the action of this IFS, the differential equation (IVP) $\frac{df}{dx}=1, \ f(0)=0$ is transformed into   $ D_J f_c(x)=\chi_C(x), \ f_c(0)=0$.} This is an alternative dynamical reinterpretation of jump derivative/differential equation in the context of Cantor set.

{\em Solution:}

For $n=0$, the IVP has the solution $f_0(x):=f(x)=x$ so that the full measure of the unit interval is obtained as $\mu[0,1]=f(1)=1$, so that the monotonic $f$ is interpreted as a measure \cite{roger}.

For n=1, we have $I_0=\frac{1}{3}(I_0)\cup G_1\cup \frac{1}{3}(2+I_0)$, where $G_1=(\frac{1}{3},\frac{2}{3})$ is the open gap created at the first application of IFS. The $0^{\rm th}$ level equation on $I_0$ gets replicated  as two self similar copies of itself on smaller closed intervals $F_1(I_0)$ and $F_2(I_0)$ in the variable $x$ and it's translates $x-\frac{2}{3}$ respectively. The replicated equation on the segment $I_1=F_1(I_0)=[0,\frac{1}{3}]$ has the form $\frac{df_1}{dx_1}=1, \ f_1=3^sf_0, \ x_1=2x, \ 0\leq x_1\leq \frac{1}{2}, \ 0 \leq x\leq \frac{1}{4}$, with identical second copy on the interval $F_2(I_0)=[\frac{2}{3},1]$ in the translated variable. After the first dissection of the open gap, the original full measure is reduced to $\frac{2}{3}$. However, in view of the above non-classical scalings of $x$ and $f$, the original full measure is retrieved if enhanced measure $\frac{1}{2}$ instead is awarded on each of the component  intervals $F_1(I_0)$ and $F_2(I_0)$. This measure enhancement is facilitated by demanding the rescaled variable $x_1$ to live in the half interval $[0,\frac{1}{2}]$, so that original variable $x$ is squeezed so as to vary in the contracted interval $[0,\frac{1}{4}]$ rather than in $[0,\frac{1}{3}]$, as in the conventional scenario.
To make the non-classical scaling consistent with the original equation, one invokes the gluing condition $3^{-s}=2^{-1}$ at the boundary point $x_1=\frac{1}{2}$ leading to the Hausdorff dimension of the Cantor set $s=\frac{\log 2}{\log 3}$. The linear contraction of the interval $[0,\frac{1}{3}]$ to $[0,\frac{1}{4}]$ is the analogue of nontrivial  squeezing that is experienced by a linear asymptotic variable in a soft model ${\bf R^*}$ as described in Sec. 2.3. The associated rescaled variable $x_1$ now parallels  the duality induced prolongation.

Alternatively, one may reinterpret the above replication  of the rescaled IVP on $I_1$ as a single copy of the original equation instead on the contracted interval $[0,\frac{1}{2}]$ as the original variable $x$ is supposed to have been approaching continuously to $\frac{1}{2}$, so that the full measure 1 is awarded on the interval $I_1$ by the rescaled measure $f_1: \ \mu[I_1]=f_1(1)=1$ for $x_1=1$. The nontrivial part in this almost trivial rescaling symmetric interpretation is the {\em choice} of the scaling constants 2 and $3^s$, as an application of the fundamental selection principle, that tells that the original measure function $f$ has also been transformed appropriately, so as to award a measure $f\mapsto  \tilde f_1(\frac{1}{2})=3^{-s}$, that must match continuously with the value $\frac{1}{2}$ that is assigned to the gap $G_1$ by the iterative ($n=1$) definition of the Cantor function $f_c$.

By self similarity, this replication proceeds over smaller scales of the form $I_n=3^nI_0$ satisfying the continuity of the boundary values across each $n$th level gap $G_n$, for instance, $2^{-n}$ with the value $3^{-ns}$ assigned on $I_n$. On the closed interval $I_n$, in the vicinity of 0, the original equation has the rescaled form $\frac{df_n}{dx_n}=1, \ f_n=3^{ns}f, \ x_n=2^n x, \ 0\leq x_n\leq 1, \ 0\leq x\leq\frac{1}{2^{n}}$ so that the full measure 1 is now retrieved by the rescaled measure $f_n$ for $s=\frac{\log 2}{\log 3}$, while the original measure function $f$ has been transformed into $n$th level definition of $f_c$ i.e. $ \tilde f_n(\frac{1}{2^n})=3^{-ns}$.

Invoking duality induced prolongation, for sufficiently large $n$, one would now write for asymptotic visible elements $x_n=2^{-n X}\approx 1, \ f_n=3^{-nsf_c(X)}\approx 1 $ (so that the transformed measure at the $n$th level of iteration now has the form $\tilde f_n=3^{-ns}\times 3^{-nsf_c(X)}$), as $x\rightarrow 0$, where $X$, the Cantor stairecase, and $f_c(X)$ are respective renormalized quantities. One obtains finally the limiting scale invariant equation, $\frac{df_c(X)}{dX}=1$ on the prolongation $\bf O$ of 0, for sufficiently small $X$ so that linearization is valid. By translation invariance and self similarity, the scale invariant equation is also valid in the prolongation of every point of $C$. Recalling that $f_c$ must be constant on a gap, the desired result $D_J f_c(x)=\chi_C(x), \ x\in [0,1]$ follows. (Q.E.D.)}
\end{example}

\begin{remark}{\rm The rescaling, the typical choices apart, considered in the above example are trivial at any finite $n$. Nontrivial scalings are realized asymptotically as $n\rightarrow \infty$ by duality structure. Judicious application of the fundamental selection principle also has a dominant role in deriving the desired result.
}
\end{remark}

 We conclude this presentation with two remarks.

1. The differentiability structure defined by the jump derivation $D_J$ allows one to treat  problems of differential equation on $C$ analogous to that in ordinary calculus. As an example, the analogue of a first order equation
\begin{equation}
\frac{dy}{dx}=f(x,y)
\end{equation}
on $\bf R$ can be modeled as a jump differential equation of the form
\begin{equation}
D_J y=\chi_C(x) f(x,y)
\end{equation}
on the Cantor set $C$ \cite{parvate}. Such a jump differential equation can be considered to describe the variation of the quantity $y$ that experiences nontrivial variation only on the Cantor set $C$, otherwise remaining constant almost everywhere in $\bf R$.  Such nontrivial variations in the neighbourhood of a point $x$ of $C$ is then described explicitly by the ordinary equation, instead,  in the renormalized effective variable $x\mapsto X$:
\begin{equation}
\frac{dy}{dX}=\tilde f(X,y)
\end{equation}
where $\tilde f$ is the normalized form of $f$ on the prolongation set of $x$.

2.  Since  jump derivation is equivalent to the ordinary derivative relative to a deformed measure of the form $dX(x)$ for an appropriate Cantor function, the entire Riemann integration theory on $\bf R$  becomes available for a large class of fractal spaces in the form of the Riemann-Stieltjes integration  relative to the deformed measure $dX(x)$.

More details of these applications will be presented in the second part of the present work II.

{\bf Acknowledgement:}
The senior author (DPD) thanks IUCAA, Pune for awarding a Visiting Associateship.

\label{pagefin}
\end{document}